\documentclass[12pt]{amsart}
\usepackage[utf8]{inputenc}
\usepackage[english]{babel}
\usepackage[T1]{fontenc}
\usepackage{amsmath}
\usepackage{amsfonts}
\usepackage{amssymb}
\usepackage{amsthm}
\usepackage{amscd}
\usepackage{soul}
\usepackage{geometry}
\usepackage{enumitem} 
\usepackage{cancel} 
\usepackage{graphicx}
\usepackage{mathtools}
\usepackage{color}
\usepackage{comment}
\usepackage{mathdots}
\usepackage{color}
\usepackage{bm}
\definecolor{marin}{rgb}   {0.,   0.3,   0.7} 
\definecolor{rouge}{rgb}   {0.8,   0.,   0.} 
\definecolor{sepia}{rgb}   {0.8,   0.5,   0.} 
\usepackage[colorlinks,citecolor=marin,linkcolor=rouge,
            bookmarksopen,
            bookmarksnumbered
           ]{hyperref}

\newcommand\N{\mathbb{N}}

\newcommand\Z{\mathbb{Z}}

\newcommand\R{\mathbb{R}}
\newcommand\C{\mathbb{C}}

\newcommand{\dd}{\mathrm{d}}

\newcommand{\enstq}[2]{\left\{#1~\middle|~#2\right\}}

\renewcommand{\Re}{\operatorname{Re}}
\newcommand{\Norm}[2]{\|#1\|\left.\vphantom{T_{j_0}^0}\!\!\right._{#2}}  
\newcommand\dr{\mathrm{D}_+}
\newcommand\dl{\mathrm{D}_-}
\newcommand\Id{\mathrm{Id}}
\newcommand\lc{\mathrm{lc}}

\newenvironment{proofof}[2]{\paragraph{\textit{ Proof of #1 #2.}}}{\hfill$\square$}

\newtheorem{theorem}{Theorem}
\newtheorem{proposition}{Proposition}
\newtheorem{corollary}{Corollary}
\newtheorem{lemma}{Lemma}

\newtheorem{remark}{Remark}
\theoremstyle{remark}

\title[Discrete Quantum Harmonic Oscillator and Kravchuk transform]{Discrete Quantum Harmonic Oscillator and Kravchuk Transform}

\author{Quentin Chauleur}
\address{INRIA Lille, Univ Lille \& Laboratoire Paul Painlevé,
CNRS UMR 8524 Lille, Cité Scientifique, 59655 Villeneuve-d'Ascq, France. }
\email{Quentin.Chauleur@math.cnrs.fr}

\author{Erwan Faou}
\address{INRIA Rennes, Univ Rennes \& Institut de Recherche Math\'ematiques de Rennes,
CNRS UMR 6625 Rennes, Campus Beaulieu F-35042 Rennes Cedex, France. }
\email{Erwan.Faou@inria.fr}

\begin{document}

\maketitle

\begin{abstract}
We consider a particular discretization of the harmonic oscillator which admits an orthogonal basis of eigenfunctions called Kravchuk functions possessing appealing properties from the numerical point of view. We analytically prove the almost second-order convergence of these discrete functions towards Hermite functions, uniformly for large numbers of modes. We then describe an efficient way to simulate these eigenfunctions and the corresponding transformation. We finally show some numerical experiments corroborating our different results.
\end{abstract}

%\tableofcontents

\section{Introduction}
Let us consider the harmonic oscillator operator, for $x \in \R^d$, 
\begin{equation} \label{harmonic_oscillator}
  H = -\Delta + |x|^2. 
\end{equation}
We are interested in discretizing this operator on a uniform grid $h \Z^d$ 
where $h>0$ denotes the stepsize of the grid. 
The harmonic oscillator appears in a lot of natural contexts, in particular as a fundamental model in quantum mechanics, and its well-known spectral properties in the whole space $\R^d$ make it a primary example of unbounded operators on Hilbert spaces.

Hermite functions are eigenfunctions of the operator $H$. They are given by the expressions, for $d = 1$, 
\begin{equation} \label{hermite_functions}
 \psi_n(x) := \frac{1}{\pi^{\frac14} 2^{\frac{n}{2}}\sqrt{n!} } e^{-\frac{x^2}{2}} H_n(x), \quad n \geq 0, 
\end{equation}
where $H_n(x)$ are the Hermite polynomials defined by the relation
\begin{equation} \label{hermite_polynomials_recurrence}
 H_{n+1}(x)= 2x H_n(x)-2n H_{n-1}(x), \quad H_0=1, 
\end{equation}
with $H_{n} = 0$ for $n < 0$. 
The set $\{ \psi_n\}_{n \in \N}$ forms a basis of $L^2(\R)$ 
satisfying the relation 
\[\forall\, n,m \in \N, \quad \int_{\R} \psi_n(x) \psi_m(x) dx = \delta_{n,m},   \]
where $\delta_{n,m}$ denotes the Kronecker symbol, and 
\[
H \psi_n = (2n +1) \psi_n.\]
Any function of $L^2(\R)$ can thus be written 
\begin{equation}
\label{decompo}
f = \sum_{n \geq 0} c_n (f) \psi_n, \quad c_n(f) = \int_{\R} f(x) \psi_n(x) \dd x,
\end{equation}
and we can generalize these properties to any dimension $d$ by tensorization. 
 The Hermite functions can also be expressed using the raising an lowering operators 
 $$
L =  \partial_{x} + x, \quad \mbox{and} \quad 
R = - \partial_{x} + x
 $$ 
 for which we have $(L f,g)_{L^2(\R)} = (f,R g)_{L^2(\R)}$ for any functions $f$ ang $g$, and  
 \[
 H =  \frac{1}{2}(RL + LR), \quad L \psi_n = \sqrt{2n} \,\psi_{n-1},\quad\mbox{and}\quad  R \psi_{n} = \sqrt{2n+2} \,\psi_{n+1}. 
 \]
The discretization of the operator $H$ poses the question of the transfer of the previous properties to discrete operators in a global perspective of geometric numerical integration or for robustness and stability arguments.  
For example the use of discrete Fourier transform will induce naturally a truncation on a large grid which has in general serious drawbacks. The same phenomenon appears in the classical discretization by finite differences. Also, the eigenvalues of the operator will be distorted in ways that are difficult to estimate, see for instance \cite{strikwerda2004,bernier2019FD}. It is also not clear at all if the natural hierarchy given by the operators $L$ and $R$ is preserved by space discretization. 

Another option would be to use spectral methods and Gauss-Hermite quadrature, but the computations of the roots of Hermite polynomials as well as the associated quadrature weights appears to be quite computationally expensive in practice \cite{funaro1992,shen2000}, and moreover, the solution is not evaluated on a standard regular grid $h\Z$ making  delicate the possible combination with other type of operator discretization.

The goal of this paper is to revitalize, amongst all possible discretizations by finite differences, the operator and functions associated with the Kravchuk polynomials, named after the Ukrainian mathematician Mikhailo Pylypovych Kravchuk\footnote{Note that a  writing difference subsists in the literature between "Kravchuk" or "Krawtchouk" polynomials, which both refer to the same mathematical object. We adopt in this paper the transcription "Kravchuk"  as commonly employed in most physical contexts, contrary to the transliteration "Krawtchouk" he may have used when writing in french, which seems to be mostly adopted in the combinatorics and probabilistic literature.}. These polynomials are well documented in the existing literature, see in particular \cite{nikiforov1991}, and they appear in discrete quantum mechanics \cite{lorente2001}, digital signal processing \cite{stobinska2019},  and coding theory \cite{levenshtein1995}, or probability in the context of multinomial distribution \cite{griffiths2014}. However, to the best of the authors knowledge, none of these properties has yet been exploited in the context of numerical analysis of quantum systems, while several important features make them {\em a priori} very appealing and worth deserving more elaborate studies. Let us summarize the main advantage of the Kravchuk functions denotes by $\varphi_{n,h}$ and defined on a regular grid $h\Z$ for $h > 0$: 
\begin{itemize}
\item They diagonalize on a regular grid a discrete tridiagonal operator $H_h$ similar to the classical order $2$ finite difference scheme. 
\item The eigenvalue associated with the Kravchuk functions are {\em exactly} the one of the Harmonic oscillator: $H_h \varphi_{n,h} = (2n+1) \varphi_{n,h}$. This isospectral diagonalization, which is usually a feature reserved to spectral methods, is very promising in particular in the numerical approximation of nonlinear evolution equations of Schr\"odinger form where discrete resonances are essential. 
\item The Kravchuk functions form an orthonormal set defined on a discrete finite grid for the standard discrete scalar product.   
\item They uniformly approximate the Hermite functions when $h \to 0$, and thus the Hermite coefficients. 
\item Finally, the computation of the Kravchuk coefficients corresponding to the Hermite coefficients (the {\em Kravchuk transform}) can be reduced to the multiplication by the exponential of a skew hermitian tridiagonal matrix. 
\end{itemize}

Each of these point would deserve complete numerical and analytical study, but we believe that the use of these polynomials could be particularly appealing in nonlinear situations, for ground states computing, or for equations coupled with operators naturally defined on regular grid. As a very first example of result, we consider in Section \ref{time_dependent} the discretization of the time dependent Schr\"odinger equation 
\[
i \partial_t \psi = H\psi
\]
 by the Kravchuk operator, and we obtain a global convergence in time as an immediate consequence of the isospectral nature of the discretization and our uniform bounds. In Theorem \ref{globalbound} we show that if $f$ is a given smooth function, and $\psi(t,x) = e^{-it H}f$, then we can construct a solution $\psi_h(t)$ to the equation 
 \begin{equation}
 \label{eqtimeh}
 i \partial \psi_h = H_h \psi_h
 \end{equation}
such that 
\[
\| \pi_h \psi(t,\cdot) - \psi_h(t, \cdot) \|_{\ell^2(h \Z)} \leq \epsilon(h)
\]
where $\epsilon(h)\to 0$ in a way depending on the smoothness of $f$, but where this estimate holds uniformly in time. 
We also note that the solution of \eqref{eqtimeh} can be obtained by the computation of the exponential of a skew-hermitian tridiagonal matrix, which can be easily done by using Pad\'e approximations, see for instance \cite{moler2003}.

\section{Main results}

Let $N \in \N^*$ be an even integer and $h = \sqrt{2}N^{-\frac12}$. 
We define the scaled Kravchuk polynomials by the relation 
\begin{equation}
\label{eqKravchuk} k_{n+1,h}(x) =2x k_{n,h}(x) - 2n \left( 1 - h^2 \Big(\frac{n-1}{2}\Big) \right) k_{n-1,h}(x), \quad k_{0,h} = 1, 
\end{equation}
with the convention $k_{n,h} = 0$ for $n < 0$ (we will moreover prove that $k_{n,h} = 0$ for $n > N$). 
We denote the finite set
\begin{equation}
\label{eq:Ah}   A_h:= h\Z \cap \left[ -\frac{1}{h},\frac{1}{h} \right], 
\end{equation}
and we consider the discrete Hilbert space $\ell^2(h\Z)$ defined through the norm
\[ \| u \|^2_{\ell^2(h \Z)} = h \sum_{a \in h \Z} |u(a)|^2 \]
induced by the scalar product
\[ \langle u,v \rangle_{\ell^2(h\Z)} = h \sum_{a \in h \Z}  u(a) \overline{v}(a) \]
for $u$, $v :h \Z \rightarrow \C$. 
%Note that the maximum is for 
%$$
%h M = \frac{1}{h} \quad \Longleftrightarrow \quad 
%M = \frac{1}{h^2} = N/2. 
%$$
We define the Kravchuk functions, for $a \in h\Z$,  
\begin{equation}
\label{Kravfunc}
 \varphi_{n,h}(a) =  \alpha_{n,h} k_{n,h}(a) \sqrt{\rho_h(a)} 
 \end{equation}
 where,  with $N h^2 = 2$ and $k = \frac{1}{h^2} + \frac{a}{h}$ for $a \in A_h$, 
\begin{equation}
\label{eq:rhoha}  \rho_h(a) =  \frac{1}{h 2^{\frac{2}{h^2}}} \frac{\Gamma( 1 + \frac{2}{h^2})}{\Gamma( 1 +  \frac{1}{h^2} + \frac{a}{h}) \Gamma( 1+  \frac{1}{h^2} - \frac{a}{h})}
= \frac{1}{h 2^N} \binom{N}{k}
\end{equation}
with the $\Gamma$ function satisfying $\Gamma(1 + n) = n!$, and $\rho_h(a) = 0$ for $a \notin A_h$,  
and
\begin{equation}
\label{alphanh}
\alpha_{n,h} = \frac{1}{h^n \sqrt{n!}}\sqrt{\frac{\Gamma( 1 + \frac{2}{h^2} - n)}{\Gamma( 1 + \frac{2}{h^2})} } = \frac{1}{h^n }\sqrt{\frac{(N-n)!}{N! n!}}. 
\end{equation}

We define the following discrete operator: for all $a \in A_h$ and $u\in \ell^2(h\Z)$,
\begin{multline} \label{discrete_hamiltonian}
 H_h u(a) = - \frac{1}{h^2} \sqrt{(1+ah+h^2)(1-ah)} u(a+h) \\
- \frac{1}{h^2} \sqrt{(1-ah+h^2)(1+ah)} u(a-h) +  \left(1 + \frac{2}{h^2} \right) u(a),  
 \end{multline}
 and the lowering and raising operators 
\begin{equation}
\label{eqraisings}
\left|
\begin{array}{rcl}
\displaystyle
L_{n,h} u(a)=  \left( nh - \frac{1}{h} + a \right)u(a) + \frac{1}{h}\sqrt{(1 - ah)(1 + ah +h^2)} u(a+h),  \\[2ex]
 \displaystyle
R_{n,h} u(a)=  \left( nh - \frac{1}{h} + a \right)u(a)  +  \frac{1}{h} \sqrt{( 1 + ah)(1 - ah +h^2)} u(a-h) 
 \end{array}
 \right.
 \end{equation}
and $H_h(a) = L_n f(a) = R_n u(a) = 0$ for $a \in h\Z \backslash A_h$. Then we have the following result, which gathers and rephrases informations that can be found in \cite{nikiforov1991,szego1975}. 
\begin{theorem}
\label{Theorem1}
We have for all $h$ such that $N = \frac{2}{h^2} \in \N^*$, and all $0 \leq n,m \leq N$, 
\begin{equation}
\left|\begin{array}{l}
H_{h} \varphi_{n,h} = (2n +1) \varphi_{n,h}\\[1ex]
\langle \varphi_{n,h}, \varphi_{m,h}\rangle_{\ell^2(h\Z)} = \delta_{nm}\\[1ex]
L_{n,h} \varphi_{n,h} = \sqrt{n(2-nh^2+h^2)}\varphi_{n-1,h} \\[1ex]
R_{n,h} \varphi_{n,h} =\sqrt{(2-nh^2)(n+1)}\varphi_{n+1,h}.
\end{array}
\right.
\end{equation}
and the operator relations 
\begin{equation}
\left|\begin{array}{l}
\langle R_{n,h} u ,v \rangle_{\ell^2(h\Z)} = \langle  u , L_{n,h} v \rangle_{\ell^2(h\Z)},  \\[1ex]
\frac12( R_{n-1,h}  L_{n,h} + L_{n+1,h}  R_{n,h})  =  (1 - ah - nh^2 )H_h   + \left( ( 2 n + 1) ah + (n+1) n h^2 \right) \Id.
\end{array}
\right.
\end{equation}
\end{theorem}
We give a complete proof  of this result (up to some calculations that are left to the reader) in order to make the paper as self contained as possible. 

The second result concerns the approximation of the Hermite function and Hermite operator by the Kravchuk functions. 
We introduce the natural projection from $H^1(\R)$ to $\ell^2(h \Z)$ defined by 
\[
(\pi_h f)(a) = f(a), 
\]
as $f$ has a continuous representative, and we introduce the weighted Sobolev spaces associated with the domain of the Harmonic oscillator operator
 \[ \Sigma^n(\R) := \enstq{ \psi \in L^2(\R)}{ \| \psi \|_{\Sigma^n(\R)}:= \|\psi \|_{H^n(\R)} + \| |
\langle x \rangle^n \psi \|_{L^2(\R)} < \infty     }  \]
 for $n \geq 0$, with $\langle x \rangle^2 = 1 + x^2$.

\begin{theorem}  \label{theorem_kravchuk_oscillator}
We have the following error estimates:
\begin{itemize}
\item[(i)]
There exists constants $C$ and $N_0$ such that for all $N \geq N_0$ and $h = \sqrt{2} N^{-\frac12}$, 
and for all  $g \in \Sigma^5(\R)$, we have
 \[ \Norm{ \pi_h \circ H g - H_h \circ \pi_h g }{\ell^2(h\Z)} \leq C h^2 \| g \|_{\Sigma^5(\R)}.              \]

\item[(ii)]For all $\delta \in (0,1)$ and $\sigma \geq 0$, there exists constants $C$ and $N_0$ such that for all $N \geq N_0$ and $h = \sqrt{2} N^{-\frac12}$,
\begin{equation}
\label{convrho} \left\| \langle a \rangle^{\sigma} \left(\rho_h(a) - \frac{1}{\sqrt{\pi}} e^{- a^2}\right)\right\|_{\ell^2(h \Z)} \leq C h^{2-\delta},   
\end{equation}
and the uniform estimate
 \begin{equation}
 \label{convphi}  \forall\, n \leq \frac13 \delta |\log h|, \quad \Norm{ \langle a \rangle^{\sigma} (\varphi_{n,h} - \pi_h \psi_n)  }{\ell^2(h\Z)} \leq C h^{2-\delta}.     
 \end{equation}
\end{itemize}
\end{theorem}

The previous proposition shows that for asymptotically large modes $n \lesssim |\log h|$, the Hermite coefficients\footnote{We conjecture that the condition $n \lesssim |\log h|$ can be replaced by $n \lesssim h^{-\delta}$ but this would require a subtle analysis of the asymptotics of the Kravchuk functions, yet to be performed.}
$c_n(f) = ( f, \psi_n)_{L^2}$ are well approximated by the discrete Kravchuck coefficients 
\[
c_{n,h}(f) = \langle \pi_h f, \varphi_{n,h}\rangle_{\ell^2(h\Z)} = h \sum_{a \in A_h} \varphi_{n}(a) f(a).
\]

Our last result, which was already noted in \cite{atakishiyev1997}, shows that these coefficients can be calculated at a cost equivalent to the evaluation of the exponential of a unitary tridiagonal matrix of size $N$: 
\begin{theorem}
\label{Theorem3}
Let $N \in \N^*$,  $h = \sqrt{2} N^{-\frac12}$, and for $k \in \{0,\ldots,N\}$, let us set 
\[
\phi_n(k) = \varphi_{n,h}( a), \quad a = - \frac{1}{h} + h k.  
\]
Then we have 
\[
c_{n,h}(f) =   \sum_{k} \phi_{n}(k) F(k) \quad \Longleftrightarrow \quad C = L F
\]
with $F(k) = h f(a)$, $C = (c_{n,h}(f))_{n = 0}^N$, $F = (f(k))_{k = 0}^N$ and 
\begin{equation}
\label{eqL} L = \left( \begin{array}{cccc}  \phi_0(0) & \phi_0(1)& \hdots &   \phi_0(N) \\
							 \phi_1(0) & \phi_1(1) & \hdots &  \phi_1(N) \\
							\vdots & \vdots & \ddots &  \vdots \\
							 \phi_N(0)  &  \phi_N(1) & \hdots & \phi_N(N)  \end{array} \right).\end{equation}
Then we have 
\begin{equation}
\label{FKT}
L = e^{\frac{i\pi (N+1) }{4}} D e^{-\frac{i \pi}{4} A}  D^*
%K= e^{-\frac{i\pi N}{4}} D^* L D
\end{equation}
with 
\begin{equation}
\label{eqDA}  D=\left( \begin{array}{cccc}  
							 1 &  & &  (0) \\
							 & e^{i \frac{\pi}{2}} & & \\
							 & & \ddots & \\
							 (0) & & & e^{i \frac{\pi N}{2}}
	  						\end{array} \right), 
							\quad \mbox{and}\quad 
							A= \left(   \begin{array}{cccc}  
					  N+1 & -\beta_1 & & (0) \\
					 -\beta_1  & N+1 & \ddots &  \\
					   & \ddots & \ddots & -\beta_N  \\
					 (0) &  & -\beta_N & N+1
				 \end{array} \right), 
				 \end{equation}
where for all $1 \leq k \leq N$ \[ \beta_k=\sqrt{k \left( N-k+1\right)}. \]
\end{theorem}
Note that in practice the matrix-vector multiplication by the exponential of a unitary tridiagonal matrix, can be easily done by using Pad\'e approximations \cite{moler2003}. The complete analysis of this transform from the numerical point of view, as well as combination with highly oscillatory situations, will be the subject of further studies. 

 \section{Discrete orthogonal Kravchuk polynomials} \label{kravchuk_appendix}

In this section, we give some fundamental properties of the discrete difference theory in order to introduce the Kravchuk polynomials. Some of the statements we give in the following are also proven for a larger class of discrete orthogonal polynomials in the book of A.F. Nikiforov, S.K. Suslov and V.B. Uvarov \cite{nikiforov1991}, however they use there an heavier formalism that we try to avoid here for clearness and conciseness purposes.

We first introduce some notations. Recall that in the following, $N \in \N^*$ and $h$ are linked by the formula $h = \sqrt{2} N^{-\frac12}$, and we write $X_N = \{0,\ldots,N\}$. We define the application $\tau_h: X_N \to A_h$ (see \eqref{eq:Ah}) by the formula 
\[ \tau_h(k) = h \left( k - \frac{N}{2} \right) = h \left( k - \frac{1}{h^2} \right).   \]
We also define the binomial distribution function on the grid $X_N$ 
\begin{equation}
\label{binomial}\Pi(k)= \frac{1}{2^N} \binom{N}{k} = \frac{1}{2^N} \frac{N!}{k! (N-k)!}, 
\end{equation}
and we extend this function to $\Z$ by setting $\Pi(k) = 0$ for $k \notin X_N$. 
%
%
%In particular, we can easily compute that
%\[ \tau_h(X_N) = h \left\{ - \frac{N}{2}, \ldots, \frac{N}{2} \right\}=h\Z \cap \left[-\frac{1}{h}, \frac{1}{h} \right] =:A_h,   \]
%so in particular the estimate $0 \leq k \leq N$ implies that $ -\frac{1}{h} \leq a \leq \frac{1}{h}$. We now define the renormalized binomial distribution on $h\Z$,
%\begin{equation}
%\label{eq:rhoha}  \rho_h(a) = \color{red} \frac{1}{h} \color{black}\Pi( \tau_h^{-1}(a)) = \frac{1}{h} \Pi( \frac{1}{h^2} + \frac{a}{h})
%\end{equation}
%We also define the renormalized Kravchuk polynomials on $h \Z$,
%\[  k_{n,h}(a) = K_n(\tau_h^{-1}(a)),   \]
%and the renormalized Kravchuk functions
%\[  \varphi_{n,h}(a) = \phi_n( \tau_h^{-1}(a)). \]
%We can transpose the properties satisfied by the eigenfunctions $\phi_n$ given above for these new functions $\varphi_{n,h}$ on $h\Z$, that we state below:
The Kravchuk polynomials we consider above are of the form 
\[
\alpha_{n,h} k_{n,h}(a) = \frac{1}{d_n}K_n(\tau_{h}^{-1}(a),N), \quad \mbox{with}
\quad  \tau_{h}^{-1}(a) = \frac{1}{h^2} + \frac{a}{h},
\]
where $K_{n}(k,N)$, $n = 0,\ldots,N$ are the standard Kravchuk polynomials, and where the constant $d_n$ is equal to 
\begin{equation}
\label{eq:dn} d_n=\frac{1}{2^n} \sqrt{\frac{N!}{n!(N-n)!}},
\end{equation}
so that 
\begin{equation}
\label{kK}
k_{n,h}(a) = \frac{1}{\alpha_{n,h} d_n}K_n(\tau_{h}^{-1}(a),N) = h^n 2^n n! K_n(\tau_{h}^{-1}(a),N). 
\end{equation}
We describe now the properties of the Kravchuk polynomials $K_n(k,N)$. 

\subsection{Discrete difference operators}
Let us recall some classical properties of difference operators. We give only some hints for the proof which are essentially based on polynomial and difference calculus that are mostly left to the reader. 

Let $f : \Z \rightarrow \C$, we define the operators $\dr$ and $\dl$ from $\C^{\Z}$ to $\C^{\Z}$ by the formula
\[ \dr f(k) = f(k+1)-f(k), 
\quad \mbox{and} \quad  \dl f(k) = f(k)-f(k-1).  
 \]
They satisfy the following properties, for 
$k \in \Z$, and $f,g \in \C^{\Z}$, 
\[
\left|
\begin{array}{l}
 \dr f(k)= \dl f(k+1),\\[1ex]
\dr \dl f(k) = \dl \dr f(k) = f(k+1)+f(k-1)-2f(k), \\[1ex]
\dr \left[ f(k) g(k) \right] = f(k) \dr g(k) + g(k+1) \dr f(k), \\[1ex]
\dl \left[ f(k) g(k) \right] = f(k-1) \dl g(k) + g(k) \dl f(k). 
\end{array}
\right.
\]
We can also prove the {\em discrete Leibniz rule}
\begin{equation}\label{discrete_leibniz_rule}
    \dr^n \left[ f(k) g(k) \right] = \sum_{j=0}^n \binom{n}{j} \dr^j f(k) \dr^{n-j} g(k+j).
\end{equation}   
where $\dr^{n}$ is the $n$-th composition of the operator $\dr$. This formula can be proved by induction using the previous relation for the derivative of products, and the Pascal formula for the binomial coefficients. 
The following result is also easy to prove: 
\begin{proposition} 
Let $P: \mathbb{Z} \rightarrow \C$ be a polynomial of order $n$. Then $\dr P$ and $\dl P$ are polynomials of order $n-1$ or less.
\end{proposition}
%\begin{proof}
%We denote $P(X)= \sum_{j=0}^n a_k X^j$ a complex polynomial of degree $n$. Then, by binomial theorem, we have
%\begin{align*}
% \dr P(k) & = a_n \left[ (k+1)^n-k^n  \right] + \sum_{j=0}^{n-1} a_j \left[ (k+1)^j-k^j \right] \\
% & =  a_n \sum_{j=0}^{n-1} \binom{j}{n} k^j + \sum_{j=0}^{n-1} a_j \left[ (k+1)^j-k^j \right],  
% \end{align*}
%which is well a polynomial of degree $n-1$ or less. The same proof holds for $\dl P$.
%\end{proof}
Eventually, the binomial distribution \eqref{binomial} satisfies the following property (called {\em Pearson equation}, see \cite{nikiforov1991}):
\begin{equation} \label{pearson_eq}
\forall k \in \Z,\quad     \dr \left(  k \Pi(k)   \right) = (N-2k) \Pi(k).
\end{equation}     

\subsection{Kravchuk polynomials}
The Kravchuk polynomials (see \cite{szego1975}) are given by the following formula: for all $ k \in X_N$,
\begin{equation} \label{kravchuk_explicit_formula}
   K_n(k,N)= \frac{1}{2^n} \sum_{j=0}^n (-1)^{n-j} \binom{k}{j} \binom{N-k}{n-j}, 
\end{equation}
where $\binom{k}{j} =k(k-1) \ldots (k-j+1)/j!$ for $k \geq j \geq 1$ and $\binom{k}{0}=1$. Note that $K_n(k,N)$ can be seen as the $n$-th coefficient of the polynomial 
\begin{equation} \label{kravchuk_polynomial_formula}
 F_{k,N}(X)=\left(1+\frac{X}{2} \right)^k\left(1-\frac{X}{2}\right)^{N-k} = \sum_{n=0}^N K_n(k,N) X^n. 
 \end{equation}
%which can be easily deduced for the classical binomial theorem. Namely, if we denote $\mathcal{P}(\cdot, X^n)$ the projector of a polynomial on the $X^n$ component (of the canonical basis $(1,X,\ldots,X^N)$), we have
%\[ K_n(k,N)=\mathcal{P}(F_{k,N}(X), X^n).   \]
%The polynomials $K_n(\cdot,N)$ are well defined by the previous formula for all $0\leq n \leq N$. We now state some fundamentals properties concerning the Kravchuk discrete polynomials:
\begin{proposition}
For a polynomial $P \in \mathbb{R} \left[X \right]$, we denote by $\lc(P) \in \R$ its leading coefficient. Then, for all $0 \leq n \leq N$, $K_n(\,\cdot\, ,N)$ is a polynomial of degree $n$, with leading coefficient 
\[ \lc(K_n)=\frac{1}{n!}. \]
\end{proposition}
\begin{proof}
The proof directly comes from the fact that  for all $0\leq j \leq k$,
\[ P_j(k):=\binom{k}{j}= \frac{1}{j!} k (k-1) \ldots (k-j +1) \]
is a polynomial of degree $j$ in $k$ with leading coefficient $1/j!$, and 
\[ Q_j(k):= \binom{N-k}{n-j}=\frac{1}{(n-j)!} (N-k)(N-k-1) \ldots (N-k-n+j+1) \]
is a polynomial of degree $n-j$ in $k$, with leading coefficient $(-1)^{n-j}/(n-j)!$. Then we have
\[ \lc(K_n)= \frac{1}{2^n} \sum_{j=0}^n (-1)^{n-j} \lc(P_j) \lc(Q_j)= \frac{1}{2^n} \sum_{j=0}^n \frac{1}{j!(n-j)!} =  \frac{1}{n!}, \] 
which gives the result.
\end{proof}
For instance, the first polynomials are given by 
\[ K_0(k,N)=1, \ \ \ K_1(k,N)=k-\frac{N}{2} \ \ \ \text{and} \ \ \ K_2(k,N)= \frac{1}{2}k^2 -\frac{N}{2} k + \frac{N(N-1)}{8}.\]
Directly from the explicit formula \eqref{kravchuk_explicit_formula}, we can prove the following: 
\begin{proposition}
For all $k,n \in \left\{0, \ldots, N \right\}$,
\[ K_n \left( N-k \right) = (-1)^n K_n (k),    \]
and 
\begin{equation}
\label{titigrosminet}
 (-1)^n 2^n \binom{N}{k} K_n (k) = (-1)^k 2^k\binom{N}{n} K_k (n).    
\end{equation}
\end{proposition}
%\begin{proof}
%It directly follows from the explicit formula \eqref{kravchuk_explicit_formula} and the change of variables $j \rightarrow n-j$ into the sum:
%\begin{align*}
%    K_n \left( N-k \right) & =  \frac{1}{2^n}\sum_{j=0}^n (-1)^{n-j} \binom{N-k}{j} \binom{k}{n-j}\\
%                         & =  \frac{1}{2^n} \sum_{j=0}^n (-1)^{j} \binom{N-k}{n-j} \binom{k}{j}\\
%                         & = (-1)^n K_n \left(k \right).
%\end{align*}
%\end{proof}
%
%
%\begin{proposition} \textbf{(Symmetry property).} \label{symmetry_prop_kravchuk_polynomials} \\
%For all $k, n \in \left\{0, \ldots, N \right\}$,
%\begin{equation}
%\label{titigrosminet}
% (-1)^n 2^n \binom{N}{k} K_n (k) = (-1)^k 2^k\binom{N}{n} K_k (n).    
%\end{equation}
%\end{proposition}
%\begin{proof}
%We assume $k\leq n$ (the proof is the same if $n \leq k$), and we compute
%\begin{align*}
%2^n\binom{N}{k} K_n (k) & = \frac{N!}{\cancel{k!} \cancel{(N-k)!}} \sum_{j=0}^n (-1)^{n-j} \frac{\cancel{k!}}{j! (k-j)!} \frac{\cancel{(N-k)!}}{(n-j)!(N-k-n+j)!} \\
%& = \frac{n!}{n!} \frac{(N-n)!}{(N-n)!} N! \sum_{j=0}^n (-1)^{n-j} \frac{1}{j! (n-j)!} \frac{1}{(k-j)! (N-n-k+j)!} \\
%& = \frac{N!}{n! (N-n)!} \sum_{j=0}^n (-1)^{n-j} \binom{n}{j} \binom{N-n}{k-j}, \\
%& = \binom{N}{n} \sum_{j=0}^k (-1)^{n-j} \binom{n}{j} \binom{N-n}{k-j} + \binom{N}{n} \sum_{j=k+1}^n (-1)^{n-j} \binom{n}{j} \underbrace{\binom{N-n}{k-j}}_{=0} \\
%& = \binom{N}{n} (-1)^{n-k} \sum_{j=0}^k (-1)^{k-j} \binom{n}{j} \binom{N-n}{k-j} 
%\end{align*}
%which gives the result.
%\end{proof}
%
Finally, the polynomials $K_n(\,\cdot\, , N)$ satisfy the following difference equation:
\begin{proposition} 
For all $k,n \in \left\{0, \ldots, N \right\}$, $K_n(\,\cdot\, ,N)$ satisfies the equation
\begin{equation} \label{difference_eq_kravchuk}
  k \dr \dl K_n(k) + (N-2k) \dr K_n(k) = - 2n K_n(k)   
 \end{equation}
\end{proposition}
\begin{proof}
We only give the main points. First this equation is equivalent to 
$$
(N-k) K_n(k+1) - (N- 2n) K_n(k)+ k K_n(k-1) = 0 
$$
for all $0 \leq k \leq N$. Differentiating equation \eqref{kravchuk_polynomial_formula} with respect to $X$, we get that
\[
  F_{k,N+1}'(X) =-\frac{k}{2}F_{k-1,N}(X)+ \frac{N+1-k}{2} F_{k,N}(X). 
\]
We evaluate the coefficients of the monomial $X^n$ in the previous expression, and we find 
% We now calculate $\mathcal{P}(F_{k,N+1}'(X),X^n)$ in two ways:
% \begin{itemize}
% 	\item From classical polynomial differentiation
% 	\[\mathcal{P}(F_{k,N+1}'(X),X^n) = (n+1) \mathcal{P}(F_{k,N+1}(X),X^{n+1})= (n+1) K_{n+1}(k,N+1). \]
% 	\item Using equation \eqref{eq_star},
% 	\[\mathcal{P}(F_{k,N+1}'(X),X^n) = - \frac{k}{2} K_n(k-1,N) + \frac{N+1-k}{2} K_n(k,N).\]
% \end{itemize}
% Hence we get the following equation:
 \begin{equation} \label{eq_star_star}
 	(n+1) K_{n+1}(k,N+1) = -\frac{k}{2} K_n(k-1,N) + \frac{N+1-k}{2} K_n(k,N).
 \end{equation}
 On the other hand, using the definition of the polynomials, we have 
  \begin{align*}
  K_n(k+1,N+1)   = K_n(k,N) - \frac{1}{2}K_{n-1}(k,N).  
  \end{align*}
% \begin{align*}
%  K_n(k+1,N+1) & = \mathcal{P}\left( (1-X/2)^{k+1}(1+X/2)^{N-k},X^n \right) \\
%  			   & =  \mathcal{P}\left( (1-X/2)^{k}(1+X/2)^{N-k},X^n \right)-\frac{1}{2}\mathcal{P}\left( (1-X/2)^{k}(1+X/2)^{N-k},X^{n-1} \right) \\
%  			   & = K_n(k,N) - \frac{1}{2}K_{n-1}(k,N).  
%  \end{align*}
  Shifting $N$ to $N-1$ into this last expression, we have
  \begin{equation} \label{eq_one}
  	K_n(k+1,N)=K_n(k,N-1)-\frac{1}{2}K_{n-1}(k,N-1).
  \end{equation}
  In the same way, we get that
  \[   K_n(k,N+1) = K_n(k,N-1) + \frac{1}{2}K_{n-1}(k,N), \]
  so shifting $N$ to $N-1$ we obtain
   \begin{equation} \label{eq_two}
  	K_n(k,N)=K_n(k,N-1)+\frac{1}{2}K_{n-1}(k,N-1).
  \end{equation} 
Substracting equation \eqref{eq_one} to equation \eqref{eq_two}, we get that
\begin{equation} \label{eq_star_star_star}
K_n(k,N)-K_n(k+1,N)= K_{n-1}(k,N-1).
\end{equation}
Finally, going back to equation \eqref{eq_star_star}, shifting $n$ to $n-1$, $N$ to $N-1$ and multiplying by 2, we have
\[  2nK_n(k,N)-(N-k)K_{n-1}(k,N-1)+k K_{n-1}(k-1,N-1)=0,    \]
so using equation \eqref{eq_star_star_star} for both $K_{n-1}$ we get that
\[ 2n K_n(k,N) - (N-k) \left( K_{n}(k,N) - K_{n}(k+1,N) \right) + k \left( K_n(k-1,N)-K_n(k,N) \right) =0.  \]
Simplifying this equation, we get the result.
\end{proof}
Using \eqref{pearson_eq} we get the following Sturm-Liouville difference equation. From now on we will often omit the $N$ in the notation for the Kravchuk polynomial, and write $K_n(k) = K_n(k,N)$. 
\begin{corollary}\label{sturm_liouville_diff_eq}
For all $0 \leq n \leq N$, $K_n$ satisfies the equation
\begin{equation} \label{kravchuk_sturm_liouville}
    \dr \left[ k \Pi(k) \dl K_n(k) \right] = 2n \Pi(k) K_n(k).
\end{equation}
\end{corollary}
\begin{proof}
We multiply equation \eqref{difference_eq_kravchuk} by $\Pi(k)$ and we use the Pearson equation \eqref{pearson_eq}, so that
\begin{multline*}
    k \Pi(k) \dr \dl K_n(k) + (N-2k) \Pi(k) \dr K_n(k) - \lambda_n K_n(k) \Pi(k) \\
      = k \Pi(k) \dr \dl K_n(k) + \dl K_n(k+1) \dr \left[ k \Pi(k) \right] -  \lambda_n K_n(k) \Pi(k) \\
      = \dr \left[ k \Pi(k) \dl K_n(k) \right] -  \lambda_n K_n(k) \Pi(k)  =0.
\end{multline*} 
\end{proof}
We also have the following orthogonality property:
\begin{proposition} For $0 \leq n,m \leq N$, we have
\begin{equation} \label{kravchuk_orthogonality_eq}
    \sum_{k=0}^N K_n(k) K_m(k) \Pi(k) = \delta_{n,m} d_n^2,
\end{equation}
with the normalization constant $d_n$ given by \eqref{eq:dn}. 
\end{proposition}
\begin{proof}
%\color{blue}
%Let $n$, $m \in X_N$, so that $K_n$ and $K_m$ satisfy the Sturm-Liouville difference equations
%\[ \dr \left[ k \Pi(k) \dl K_n(k) \right] = \lambda_n \Pi(k) K_n(k),  \]
%\[ \dr \left[ k \Pi(k) \dl K_m(k) \right] = \lambda_m \Pi(k) K_m(k) \]
%for $k \in X_N$. We multiply the first one by $K_m(k)$ and the second one by $K_n(k)$, then we subtract the second one from the first, so that
%\begin{align*}
%    (\lambda_n - \lambda_m) \Pi(k) K_m(k) K_n(k) &= K_m(k) \dr \left[ k \Pi(k) \dl K_n(k)  \right] - K_n(k) \dr \left[ k \Pi(k) \dl K_m(k)    \right] \\
%    & = \dr \left[ k \Pi(k)\left( K_m(k) \dl K_n(k)- K_n(k) \dl K_m(k) \right) \right],
%\end{align*}   
%where the last equality comes from developing and identifying all the terms in the two expressions. Then, summing over $k \in X_N$, we get that
%\[  (\lambda_n - \lambda_m) \sum_{k=0}^{N} \Pi(k) K_m(k) K_n(k) = \left[   k \Pi(k) \left( K_m(k) \dl K_n(k) - K_n(k) \dl K_m(k)\right) \right]^{N+1}_0 =0 \]
%as $\Pi(N+1)=0$.  It then remains to prove that
%\[ d_n = \sqrt{ \sum_{k=0}^{N} \Pi(k) K_n(k)^2  } = \frac{1}{2^n} \sqrt{\binom{N}{n}}.    \]
%\color{black}
The proof consists in expanding the relation 
\[  \left[\left(1- \frac{X}{2} \right)\left(1-\frac{Y}{2} \right)+\left(1+\frac{X}{2} \right)\left(1+\frac{Y}{2} \right)\right]^N =2^N \left(1+\frac{XY}{4} \right)^N \]
using \eqref{kravchuk_polynomial_formula}.
%From binomial theorem we easily get that
%\[ \left[ 2\left(1+\frac{XY}{4} \right) \right]^N = 2^N \sum_{n=0}^N \binom{N}{n} \frac{X^n Y^n}{4^n}. \]
%On the other hand, 
%\begin{multline*}
% \left[ \left(1- \frac{X}{2} \right)\left(1-\frac{Y}{2} \right)+\left(1+\frac{X}{2} \right)\left(1+\frac{Y}{2} \right) \right]^N\\
%  = \sum_{k=0}^N \binom{N}{k} \left(1+ \frac{X}{2} \right)^k \left(1+ \frac{Y}{2} \right)^k \left(1-\frac{X}{2}\right)^{N-k} \left(1-\frac{Y}{2}\right)^{N-k} \\
%										= \sum_{k=0}^N \binom{N}{k} \left( \sum_{n=0}^N K_n(k) X^n \right) \left( \sum_{n=0}^N K_m(k) X^m \right) \\
%										= 2^N \sum_{n,m=0}^N \left( \sum_{k=0}^N  \frac{1}{2^N}\binom{N}{k} K_n(k) K_m(k) \right) X^n Y^m, 
%\end{multline*}
%hence we get the result by polynomial identification. 
\end{proof}
The following formula will explain the specific form of the Kravchuk transform: 
\begin{proposition}\label{kravchuk_polynomial_transform_property} 
For $0 \leq n,m \leq N$,
\[ \sum_{k=0}^N 2^k K_k(n) K_m(k) e^{\frac{i (k-m) \pi}{2}} =2^{\frac{N}{2}} e^{\frac{in \pi}{4}} e^{-\frac{i \pi N}{4}}  K_m(n,N).  \]
\end{proposition}
\begin{proof}
The proof consists in expanding the relation
\begin{multline*}
\left( 1 - \frac{X}{2} \right)^N \left(1+i \frac{1+X/2}{1-X/2}  \right)^n \left(1-i \frac{1+X/2}{1-X/2}    \right)^{N-n}\\
   = (1+i)^n (1-i)^{N-n} \left(1+i \frac{X}{2} \right)^n \left(1-i\frac{X}{2} \right)^{N-n},   
  \end{multline*}
  by using the definition of the Kravchuk polynomials. 
%
%The proof consist in expanding the equation
%\begin{multline*}
%  \left( 1+i -(1-i)\frac{X}{2} \right)^n \left(1-i-(1+i)\frac{X}{2} \right)^{N-n} \\
%   = (1+i)^n (1-i)^{N-n} \left(1+i \frac{X}{2} \right)^n \left(1-i\frac{X}{2} \right)^{N-n},   
%  \end{multline*}
%that we can compute in two ways. First, we note that
%\[   (1+i)^n (1-i)^{N-n} \left(1+i\frac{X}{2} \right)^n \left(1-i\frac{X}{2} \right)^{N-n} = 2^{\frac{N}{2}} e^{ \frac{in\pi}{4}} e^{-\frac{i \pi N}{4}} \sum_{m=0}^N K_m(n,N) (iX)^m.   \]
%On the other hand
%\begin{multline*}
%(1+i)^n (1-i)^{N-n} \left(1+i \frac{X}{2} \right)^n \left(1-i\frac{X}{2} \right)^{N-n} \\
% = \left( 1 - \frac{X}{2} \right)^N \left(1+i \frac{1+X/2}{1-X/2}  \right)^n \left(1-i \frac{1+X/2}{1-X/2}    \right)^{N-n} \\
%								    = \left( 1 - \frac{X}{2} \right)^N \sum_{k=0}^N K_k(n) \left(2i \frac{1+X/2}{1-X/2} \right)^k \\
%								    = \sum_{k=0}^N 2^k e^{\frac{i k \pi}{2}} K_k(n) \left(1+ \frac{X}{2} \right)^k \left(1- \frac{X}{2} \right)^{N-k} \\
%								    = \sum_{k=0}^N 2^k e^{\frac{i k \pi}{2}} K_k(n) \sum_{m=0}^N K_m(k) X^m \\
%								    = \sum_{m=0}^N  \left( \sum_{k=0}^N 2^k K_k(n) K_m(k) e^{\frac{i(k-m) \pi}{2}}  \right) (iX)^m,
%\end{multline*}
%so by polynomial identification we get the result.
\end{proof}
We finally give the three-term recurrence relation for the Kravchuk polynomials:
\begin{proposition}  \label{prop_kravchuk_recurrence}
For all $n \in \N^*$ and for $k \in X_N$, we have
\begin{equation} \label{kravchuk_recurrence}
 (n+1) K_{n+1}(k) = \left(k-\frac{N}{2} \right) K_n(k) - \frac{N-n+1}{4} K_{n-1}(k).    
 \end{equation}
\end{proposition}
\begin{proof}
We denote by $a_n$ and $b_n$ the leading coefficients in the expansion
\[ K_n(k)=a_n k^n + b_n k^{n-1} + \ldots, \]
and we have (see for instance \cite[p. 44]{nikiforov1991})
\[ a_n= \frac{1}{n!}, \ \ \ \text{and} \ \ \ b_n = - \frac{N}{2 (n-1)!}. \]
Let us now remark that
\[ \deg \left( K_{n+1} - \frac{a_{n+1}}{a_n} k K_n   \right) \leq n,   \]
This shows that 
\[ K_{n+1} - \frac{a_{n+1}}{a_n} k K_n = \sum_{l=0}^{n} c_l K_l,   \]
and by taking the weighted discrete scalar products with weight $\Pi(k)$ and using \eqref{kravchuk_orthogonality_eq}, we can prove that $c_l = 0$ for $l < n-1$ as $\deg(kK_l)<n$. 
%with coefficients 
%\[ c_l \langle K_l, K_l \rangle = \langle K_{n+1} - \frac{a_{n+1}}{a_n} k K_n, K_l \rangle = - \frac{a_{n+1}}{a_n} \langle K_n,k K_l \rangle. \]
%But for $l<n-1$, this last scalar product vanishes, as $\deg(kK_l)<n$. 
We obtain a relation of the form 
\[   k K_n = \alpha_n K_{n+1} + \beta_n K_n + \gamma_n K_{n-1}, \]
whose coefficients $\alpha_n$, $\beta_n$ and $\gamma_n$ can be found by the formulas
\[  \alpha_n=\frac{a_n}{a_{n+1}}, \ \ \ \beta_n= \frac{b_n}{a_n} - \frac{b_{n+1}}{a_{n+1}}, \ \ \ \gamma_n = \frac{a_{n-1}}{a_n} \frac{d_n^2}{d_{n-1}^2}   \]
by standard computations, which gives the result. 
\end{proof}

We are now going to prove the discrete counterpart of the classical Rodrigues formula for Kravchuk polynomials, as it will be useful for the derivation of the lowering and raising operators. We first need to introduce some notations and properties in the following lemma, which is a consequence of the formula for the discrete derivative of products: 
\begin{lemma} \label{lemma_pre_rodriques}
Let $n$, $m\geq 0$, and $k \in \Z$. We denote
\[  \Pi_m(k) = \Pi(k+m) \prod_{j=1}^m (k+j), \ \ \ K_n^{(m)}= \dr^m K_n \ \ \ \text{and} \ \ \ \mu_n^{(m)}=2n - 2m.    \]
Then, for all $k \in \Z$, $K_n^{(m)}$ satisfies the following Sturm-Liouville difference equation:
\[  \dr \left[ k \Pi_m(k)  \dl K_n^{(m)}(k)   \right] + \mu_n^{(m)} \Pi_m(k)  K_n^{(m)}(k) =0 . \]
\end{lemma}
\begin{proof}
Differentiating equation \eqref{difference_eq_kravchuk}, we get that
\[ \dr \left[ k \dl K_n^{(1)}(k) \right] + \dr \left[ (N-2k) K_n^{(1)}(k)    \right] + \lambda_n K_n^{(1)}(k) =0,   \]
which simplifies in
\[  k \dr \dl K_n^{(1)}(k) + (N-2k-1) \dr K_n^{(1)}(k) + (\lambda_n -2)  K_n^{(1)}(k)=0. \]
Applying $m$ times this computation, we can prove by induction that $\forall m \geq 0$, that $K_n^{(m)}$ satisfies the following difference equation:
\[  k \dr \dl K_n^{(m)}(k) + (N-2k-m) \dr K_n^{(m)}(k) + (\lambda_n - 2m)  K_n^{(m)}(k)=0. \]
We are now going to show that $\Pi_m$ satisfies a Pearson-type equation. In fact,
\begin{align*}
    \dr \left[ k \Pi_m(k)  \right] & = (k+1) \Pi(k+1+m) \prod_{j=1}^m (k+1+j) -  k \Pi(k+m) \prod_{j=1}^m (k+j) \\
    & =  \frac{1}{2^N} \left(\prod_{j=1}^m ( k + j) \right) \left[ (k+1+m) \binom{N}{k+1+m} - k \binom{N}{k+m}  \right] \\
    & =  \frac{1}{2^N}\left(\prod_{j=1}^m ( k + j) \right)\left[ (N-k-m) \binom{N}{k+m} - k \binom{N}{k+m} \right] \\
    & = (N-2k-m) \Pi_m(k).
\end{align*}
We then just have to reproduce the computation of the proof of Corollary \ref{sturm_liouville_diff_eq} to get the result.
\end{proof}

\begin{proposition} \textbf{(Rodrigues formula).} 
%For all $n \in \N$, we denote 
%\[ c_n = \frac{(-1)^n}{2^n n!}. \]
For $k \in X_N$, we have
\begin{equation} \label{discrete_rodriques}
 K_n(k)\Pi(k) =\frac{(-1)^n}{2^n n!} \dr^n \left[ \Pi(k) \prod_{j=0}^{n-1} (k-j)  \right].
 \end{equation}
\end{proposition}
\begin{proof}
With the previous notations, it is a direct consequence of Lemma \ref{lemma_pre_rodriques} to see that
\[
    \Pi_m(k) K_n^{(m)}(k)  = -\frac{1}{\mu_m^{(n)}} \dr \left[  k \Pi_m(k) \dl K_n^{(m)}(k) \right]  = -\frac{1}{\mu_m^{(n)}} \dl \left[ \Pi_{m+1}(k) K_n^{(m+1)}(k) \right]. 
\]%
%\begin{align*}
%    \Pi_m(k) K_n^{(m)}(k) & = -\frac{1}{\mu_m^{(n)}} \dr \left[  k \Pi_m(k) \dl K_n^{(m)}(k) \right] \\
%    & = -\frac{1}{\mu_m^{(n)}} \dl \left[  (k+1) \Pi_m(k+1) \dr K_n^{(m)}(k) \right] \\
%    & = -\frac{1}{\mu_m^{(n)}} \dl \left[ \Pi_{m+1}(k) K_n^{(m+1)}(k) \right]. 
%\end{align*}
By induction we get that
\[ \Pi(k) K_n(k) = \prod_{j=0}^{n-1} \left( -\frac{1}{\mu_j^{(n)}} \right)  \dl^n \left[ \Pi_n(k) K_n^{(n)}(k)   \right]. \]
As $K_n$ is polynomial of degree $n$, $K_n^{(n)}= \frac{1}{n!}$ is a constant, so we calculate directly,  
\[ \Pi(k) K_n(k) = c_n \dl^n \Pi_n(k)  = \frac{(-1)^n}{2^n n!} \dr^n \left[ \Pi(k) \prod_{j=0}^{n-1} (k-j) \right],  \]
which gives the result.
\end{proof}
\begin{corollary}\label{kravchuk_diff_eq_order_1}
For all $n \in \N$ and $k \in X_N$,
\[  2(n+1) K_{n+1}(k) = (n+2k-N) K_n(k) - k \dl K_n(k).   \]
\end{corollary}
\begin{proof}
With $c_n = \frac{(-1)^n}{2^n n!}$, by the Rodrigues formula, we have that
\[ K_{n+1}(k) \Pi(k)= c_{n+1} \dl^{n+1} \Pi_{n+1}(k) = c_{n+1} \dl^n \left[ \dr \Pi_{n+1}(k-1) \right] .  \]
Noticing that we can compute
\[ \dr \Pi_{n+1}(k-1) =\dr \left[ \Pi(k+n) \prod_{j=1}^{n+1} (k-1+j) \right] = \dr \left[k \Pi_n(k) \right] = (N-2k-n) \Pi_n(k),    \]
we see that
\begin{align*}
 K_{n+1}(k) \Pi(k) & = c_{n+1} \dl^n \left[ (N-2k-n) \Pi_n(k)  \right] \\
&= c_{n+1} \dl^{n-1} \left[  (N-2k-n) \dl \Pi_n(k)  -  2 \dl \Pi_n(k-1)  \right]\\
&  = c_{n+1} \left( (N-2k-n) \dl^n \left[ \Pi_n(k) \right] -2n \dl^{n-1} \left[ \Pi_n(k-1) \right]  \right)  
 \end{align*}
as we can prove easily by induction. Since 
\[ \dl K_n(k) = K_n^{(1)} (k-1) = \frac{c_n}{k \Pi(k)} \dl^{n-1} \left[ \Pi_n(k-1) \right],    \]
we obtain 
$$
 K_{n+1}(k) = \frac{c_{n+1}}{c_n}(N - 2k - n ) K_n(k)  - \frac{c_{n+1}}{c_n} \dl K_n(k)
$$
which shows the result. 
\end{proof}

\subsection{Kravchuk functions}
We now define the Kravchuk functions $(\phi_n)_n$, such that for all $k \in X_N$,
\begin{equation}
\label{kravfon}  \phi_n(k) = \frac{1}{d_n} K_n(k) \sqrt{\Pi(k)},  \end{equation}
and $\phi_n(k)=0$ elsewhere. From \eqref{kravchuk_orthogonality_eq} we immediately get that the sequence $(\phi_n)_n$ is orthogonal for the discrete scalar product on $X_N$: 
\begin{proposition}\label{prop_orthogonality_kravchuk_functions}
For all $0 \leq n,m \leq N$,
\[  \langle \phi_n, \phi_m \rangle_{\ell^2(\Z)} = \sum_{k=0}^N \phi_n(k) \phi_m(k) = \delta_{n,m}.  \]
\end{proposition}

Several properties of the Kravchuk polynomials can of course be easily passed to the Kravchuk functions and we omit the proof of the following formulas: 

\begin{proposition}\label{prop_transform_kravchuk_functions} %\label{prop_symmetry_kravchuk_functions}
For all $0 \leq k,n \leq N$,
\[  (-1)^k \phi_n(k)= (-1)^n \phi_k(n);  \]
%\end{proposition}
%\begin{proof}
%Using Proposition \ref{symmetry_prop_kravchuk_polynomials}, we easily compute
%\begin{align*} 
%\phi_n(k) & = \frac{1}{d _n}  \sqrt{\Pi(k)} K_n(k)\\
%		  & = \frac{2^n}{ \sqrt{\binom{N}{n}}} \frac{(-1)^{k-n} 2^k}{2^n \binom{N}{k}} \binom{N}{n} K_k(n) \sqrt{\frac{1}{2^N} \binom{N}{k}}\\
%					   & = (-1)^{k-n} \frac{2^k}{\sqrt{\binom{N}{k}}} K_k(n) \sqrt{\frac{1}{2^N} \binom{N}{n} } =(-1)^{k-n} \phi_k(n),
%\end{align*}
%hence we get the result.
%\end{proof}
%
%\begin{proposition}\textbf{(Transform property of Kravchuk functions).} 

and for all $0 \leq m,n \leq N$,
\[  \phi_m(n) e^{\frac{in \pi}{4}} e^{-\frac{i \pi N}{4}} =   \sum_{k=0}^N \phi_k(n) \phi_m(k) e^{\frac{i (k-m) \pi}{2}}. \]
\end{proposition}

\subsection{Kravchuk oscillator}

As Hermite functions for the harmonic oscillator, Kravchuk functions are eigenfunctions for a particular discrete operator denoted $\mathcal{H}$, which also admits a ladder operator description that we explicit in the following.

\begin{proposition} \textbf{(Kravchuk oscillator on $\Z$).} \label{prop_eigenfunctions_X_N}
We define the operator $\mathcal{H}$ for $f \in \ell^2(\Z)$  by 
\[ \mathcal{H} f(k)= \sqrt{(k+1)(N-k)} f(k+1) +  \sqrt{k(N-k+1)} f(k-1)-N f(k),    \]
for $k \in X_N$, and $\mathcal{H} f(k)=0$ if $k \notin X_N$. Then
\[ \mathcal{H}\phi_n=-2n \phi_n.    \]
\end{proposition}
\begin{proof}
It comes directly from multiplying equation \eqref{difference_eq_kravchuk} by $d_n^{-1}\sqrt{\Pi(k)}$, and noticing that formally
\[  \frac{\Pi(k)}{\Pi(k+1)} = \frac{k+1}{N-k} \ \ \ \text{and} \ \ \ \frac{\Pi(k)}{\Pi(k-1)}=\frac{N-k+1}{k}. \]
\end{proof}

\begin{proposition} \textbf{(Lowering and raising operators).} \\
We define the operators $\mathcal{L}$ and $\mathcal{R}$ acting on $\ell^2(\Z)$ by 
\[
\left|
\begin{array}{l} \mathcal{L} f(k)=  (k-N)f(k) + \sqrt{(N-k)(k+1)} f(k+1), \\[1ex]
\mathcal{R} f(k)=  (k-N) f(k) + \sqrt{k(N-k+1)} f(k-1),
\end{array}
\right.
\]
for $k \in X_N$, $\mathcal{L} f(k)= \mathcal{R} f(k)=0$ if $k \notin X_N$, and we denote 
\[ \mathcal{L}_n = \mathcal{L} +n \Id \ \ \ \text{and} \ \ \ \mathcal{R}_n = \mathcal{R} +n \Id,   \]
for all $n \geq 0$.
Then, for all $0 \leq n \leq N$, we have
\begin{equation} \label{lowering_operator_kravchuk_eq}
 \mathcal{L}_n \phi_n = \sqrt{n(N-n+1)}\phi_{n-1}   
\quad\mbox{and}\quad
 \mathcal{R}_n \phi_n =\sqrt{(N-n)(n+1)}\phi_{n+1}.   
\end{equation}
In particular, for all $1 \leq n \leq N$, we have
\[  \mathcal{R}_{n-1}  \mathcal{L}_n \phi_n  =  n(N-n+1)\phi_n \quad\mbox{and}\quad \mathcal{L}_{n+1}  \mathcal{R}_n \phi_n  = n(N-n+1) \phi_n. \]
%The operators $(\mathcal{L}_n)_n$ are called \textbf{lowering operators} and  the operators $(\mathcal{R}_n)_n$ \textbf{raising operators}. 
\end{proposition}
\begin{proof}
We compute
\begin{align*}
    d_n \mathcal{R}_n \phi_n(k) & = (k+n-N) K_n(k) \sqrt{\Pi(k)} + \sqrt{k(N-k+1)} K_n(k-1) \sqrt{\Pi(k-1)} \\
    & = (k+n-N) K_n(k) \sqrt{\Pi(k)} + k \sqrt{\Pi(k)} K_n(k-1) \\
    & =  (n+2k-N) K_n(k) \sqrt{\Pi(k)} -\sqrt{\Pi(k)} k \dl K_n(k) \\
    & = - 2 (n+1) K_{n+1}(k) \sqrt{\Pi(k)}     
\end{align*}
using Corollary \ref{kravchuk_diff_eq_order_1},
so \[  \mathcal{R}_n \phi_n(k) = -2(n+1) \frac{d_{n+1}}{d_n} \phi_{n+1}(k)=\sqrt{(N-n)(n+1)} \phi_{n+1}(k).   \]
The calculation for the operator $\mathcal{L}_n$ is entirely similar. 
\end{proof}

\begin{proposition} \textbf{(Self-adjointness of $\mathcal{H}$).} \label{kravchuk_oscillator_self_ajoint}\\
For all $n \geq 0$, the operators $\mathcal{L}_n$ and $\mathcal{R}_n$ are adjoint, namely for all $u$, $v \in \ell^2(\Z)$,
\[ \langle \mathcal{R}_n u ,v \rangle_{\ell^2(\Z)} = \langle  u , \mathcal{L}_n v \rangle_{\ell^2(\Z)}.    \]
In particular, the operator $\mathcal{H}$ is self-adjoint.
\end{proposition}
\begin{proof}
We compute 
\[ \langle \mathcal{R}_n u ,v \rangle = \sum_{k=0}^N (k-N+n) u(k) v(k) + \sum_{k=0}^N \sqrt{k(N-k+1)} u(k-1) v(k). \]
In order to make the change of variable $k \mapsto k+1$, we need to remark that the quantity
\[  \sqrt{k(N-k+1)} = 0 \ \text{when} \ k=0   \]
and 
\[  \sqrt{(k+1)(N-k)} = 0 \ \text{when} \ k=N,   \]
so we can rewrite the secund sum
\begin{align*}
\sum_{k=0}^N \sqrt{k(N-k+1)} u(k-1) v(k) & = \sum_{k=1}^N \sqrt{k(N-k+1)} u(k-1) v(k) \\
& = \sum_{k=0}^{N-1} \sqrt{(k+1)(N-k)} u(k) v(k+1) \\
& = \sum_{k=0}^{N} \sqrt{(k+1)(N-k)} u(k) v(k+1).
\end{align*}
The proof for $\mathcal{H}$ is a consequence of the factorization of $\mathcal{H}$ of the next proposition. \end{proof}

\begin{proposition} \textbf{(Factorization of $\mathcal{H}$).}\\
For all $0 \leq n \leq N$ and $k \in X_N$, we have
\[  \mathcal{R}_{n-1}  \mathcal{L}_n  =  (k+n-1-N) \left( \mathcal{H} + n \Id \right) + nk \Id\]
and
\[  \mathcal{L}_{n+1}  \mathcal{R}_n  =  (k+n+1-N)  \left( \mathcal{H} + n \Id \right) + (nk+N)\Id, \]
and therefore
\begin{multline*}
\frac12( \mathcal{R}_{n-1}  \mathcal{L}_n + \mathcal{L}_{n+1}  \mathcal{R}_n)  \\
=  (k+n-N)(\mathcal{H}-  1)  + \left(( n + 1) (k+n-N) +   nk + \frac{N}{2} \right) \Id. 
\end{multline*}
\end{proposition}
\begin{proof}
We compute
\begin{multline*}
 \mathcal{L}_{n+1}  \mathcal{R}_n f(k)  =   (k+n+1-N) \mathcal{R}_n f(k) +\sqrt{(N-k)(k+1)} \mathcal{R}_n f(k+1) \\
 = (k+n+1-N) (k+n-N) f(k) + (k+n+1-N) \sqrt{k(N-k+1)} f(k-1) \\
 + (k+1+n-N) \sqrt{(N-k)(k+1)}  f(k+1) +(N-k)(k+1) f(k) \\
 = (k+n+1-N) \left( \mathcal{H} + n \Id \right) f(k)  + k(k+n+1-N)f(k) + (k+1)(N-k) f(k)  \\
 = (k+n+1-N) \left( \mathcal{H} + n \Id \right) f(k) + (N + nk) f(k).
\end{multline*}
We compute $\mathcal{R}_{n-1}  \mathcal{L}_n f(k)$ the same way:
\begin{multline*}
 \mathcal{R}_{n-1} \mathcal{L}_{n}   f(k)  =   (k+n-1-N) \mathcal{L}_n f(k) +\sqrt{k(N-k+1)} \mathcal{L}_n f(k-1) \\
 = (k+n-1-N) (k+n-N) f(k) + (k+n-1-N) \sqrt{(k+1)k(N-k)} f(k+1) \\
 + (k+n-1-N) \sqrt{(N-k+1)k}  f(k+1) +(N-k+1)k f(k) \\
 = (k+n-1-N) \left( \mathcal{H} + n \Id \right) f(k)  + k(k+n+1-N)f(k) + (k+1)(N-k) f(k)  \\
 = (k+n-1-N) \left( \mathcal{H} + n \Id \right) f(k) + nk f(k).
\end{multline*}
and thus 
\begin{multline*}
\frac12( \mathcal{R}_{n-1}  \mathcal{L}_n + \mathcal{L}_{n+1}  \mathcal{R}_n)   =  (k+n-N) \left( \mathcal{H} + n \Id \right) + \left(nk + \frac{N}{2} \right)\Id   \\
=  (k+n-N)(\mathcal{H}-  1)  + ( n + 1) (k+n-N)\Id +   \left(nk + \frac{N}{2} \right) \Id. \end{multline*}
\end{proof}
 \subsection{Kravchuk functions and Kravchuk oscillator on $h\Z$}
 Theorem \ref{Theorem1} is now a consequence of the previous results, using \eqref{kK}. For instance, we get  that  for all $n$, $m \geq 0$,
\[  h \sum_{a \in h \Z} K_{n}(\tau_h^{-1}(a)) K_{m}(\tau_h^{-1}(a)) \rho_h(a) = \delta_{n,m} d_n^2,  \]
as we have $\rho_h(a) = \frac{1}{h}\Pi(\tau_h^{-1}(a))$, see \eqref{eq:rhoha}. We can also check that with formula \eqref{kravfon}, we have
\[
\varphi_{n,h}(a) = \frac{1}{\sqrt{h}}\phi_{n}(\tau_h^{-1}(a)). 
\]
This directly gives that
\[   h  \sum_{a \in h \Z} \varphi_{n,h}(a) \varphi_{m,h}(a) = \delta_{n,m}.   \]
%
%\begin{proof}
%This property is a direct consequence of the orthogonality of the Kravchuk polynomials with respect to the density $\Pi$ on $X_N$ stated in Proposition \ref{kravchuk_orthogonality_property}. In fact, making the change of variable $k=\tau_h^{-1}(a)$ in equation \eqref{kravchuk_orthogonality_eq}, we get that
%\[ \color{red} h \color{black}  \sum_{a \in A_h} k_{n,h}(a) k_{m,h}(a) \rho_h(a) = \delta_{n,m} d_n^2 \]
%for all $n$, $m \geq 0$, hence the result recalling that $ \rho_h(a)=0$ when $a \notin A_h$.
%\end{proof}
%\color{red}Ici je change la definition de $\rho_h$ ce qui permet d'avoir la bonne normalisation dans le produit scalaire\color{black}
%
Then we define for $f \in \ell^2(h\Z)$, 
\[
\left|
\begin{array}{l}
H_h f(a) = (-\mathcal{H} + 1)f \circ \tau_h^{-1}(a),\\
L_{n,h} f(a)  = h \mathcal{L}_n f \circ \tau_h^{-1}(a),\\
R_{n,h} f(a)  = h \mathcal{R}_n f \circ \tau_h^{-1}(a),
\end{array}
\right.
\]
and from the expression of these operators in variable $a$, we deduce 
\eqref{discrete_hamiltonian} and \eqref{eqraisings}, and 
Theorem \ref{Theorem1} follows. Eventually, Proposition~\ref{prop_kravchuk_recurrence} shows that 
for all $n \in \N^*$, and for all $a \in h\Z$,
\[ (n+1)K_{n+1}(\tau_h^{-1}(a))  = \frac{a}{h} K_{n}(\tau_h^{-1}(a)) - \frac{1}{4} (N-n+1)K_{n-1}(\tau_h^{-1}(a)) .   \]
As $k_{n,h}(a) = h^n 2^n n! K_{n}(\tau_h^{-1}(a))$, we calculate that for all $n \in \N^*$ and $a \in h\Z$,% denoting $c_n=h^n 2^n n! $, and multiplying the recurrence equation on $(k_{n,h})_n$ by $h c_n$, we get that
%\[ h(n+1)c_n k_{n+1,h}(a) = a c_n k_{n,h}(a) - \frac{h}{4} c_n (N-n+1) k_{n-1,h}(a).\]
%Then, we calculate that $h(n+1)c_n = \frac{1}{2} c_{n+1}$ and 
%\[  \frac{h}{4} c_n (N-n+1) =  \frac{1}{2h} c_n \left( 1 -\frac{n-1}{N}\right) =  n c_{n-1} \left( 1- h^2\frac{n-1}{2} \right). 
%\]
%so we get the result.
%
\[ k_{n+1,h} =2a k_{n,h} - 2n \left( 1 - h^2 \Big(\frac{n-1}{2}\Big) \right) k_{n-1,h},  \]
which is \eqref{eqKravchuk}.

%\begin{proposition} \textbf{(Eigenfunction property).}\\
%For all $a \in A_h$, we recall that
%\begin{multline*}
% H_h g(a) =  -  \left( \frac{1}{h^2} \sqrt{(1+ah+h^2)(1-ah)} g(a+h) \right. \\
% \left. + \frac{1}{h^2} \sqrt{(1-ah+h^2)(1+ah)} g(a-h)- \frac{2}{h^2} g(a) \right)\mathbf{1}_{ -\frac{1}{h} \leq a \leq \frac{1}{h} }.  
%\end{multline*}
%Then, the finite operator $H_h$ has $N+1$ eigenfunctions (where $N=2/h^2$) such that for all $0 \leq n \leq N$,
%\[ H_h \varphi_{n,h}= \lambda_n \varphi_{n,h} \ \ \ \text{with} \ \ \ \lambda_n=2n,  \]
%and is self-adjoint on the Hilbert space $\ell^2(h\Z)$, which means that
%\[  \langle H_h u, v \rangle_{\ell^2(h\Z)} = \langle u, H_h v \rangle_{\ell^2(h\Z)}   \]
%for all $u$, $v \in L^2(h \Z)$.
%\end{proposition}
%\begin{proof}
%Once again, these properties directly come from the change of variable $k=\tau_h^{-1}(a)$ in respectively Proposition \ref{prop_eigenfunctions_X_N} and Proposition \ref{kravchuk_oscillator_self_ajoint}.
%\end{proof}

 \section{Convergence of Kravchuk oscillator} \label{kravchuk_oscillator_section}
  
 We first show some inequalities between discrete and continuous Sobolev spaces, which will be needed for the proof of Theorem \ref{theorem_kravchuk_oscillator}.
 
 \begin{lemma} \label{lemma_L2_H1}
 Let $ g \in H^1(\R)$, then $ \pi_h g \in \ell^2(h\Z)$ and
 \begin{equation}
 \label{eq1} \| \pi_h g \|_{\ell^2(h\Z)} \leq 2 \| g \|_{H^1(\R)}   
 \end{equation}
 as soon as $h \leq \sqrt{2}$.
 \end{lemma}
 \begin{proof}
 Let $a = jh \in h\Z$ with $j \in \Z$. We take $x \in \left[ a, a+h \right]$. As $g \in H^1(\R)$, we can write that
 \[ (\pi_h g)(a) = g(x) - \int_{a}^x \partial_x g(y) \dd y,    \]
 so taking the square on this equation, and by Cauchy-Schwarz inequality and Young's inequality for products, we see that
 \[ |\pi_h g(a) |^2 \leq 2 |g(x)|^2 + 2 (x-a)   \int_{a}^{a+h}  \left| \partial_x g(y)\right|^2 \dd y.   \]
 Then, integrating (with respect to $\dd x$) between $a$ and $a+1$ and summing over $a \in h\Z$, we obtain that
 \[  h \sum_{a \in h\Z} | \pi_h g(a)|^2 \leq 2 \int_{\R} |g(x)|^2 \dd x  + h^2 \int_{\R} \left| \partial_x g(y)\right|^2 \dd y, \]
 so 
 \[ \| \pi_h g \|_{\ell^2(h\Z)}^2 \leq  2 \| g \|_{L^2(\R)}^2 + h^2 \| g \|_{\dot{H}^1(\R)}^2,    \]
 hence we get the results as soon as $h \leq \sqrt{2}$.
 \end{proof}
 
 \begin{remark}
 Note that in our case, as $h=\sqrt{2/N}$ with $N \geq 1$, we always fulfill the condition $h \leq \sqrt{2}$. 
 \end{remark}
 
 We now state a lemma that will be useful in the proof of Theorem \ref{theorem_kravchuk_oscillator}:
 
 \begin{lemma} \label{weighted_discrete_sobolev_prop}
For all $\alpha$, $\beta \in \N$ and $n$ such that $\alpha+\beta\leq n-1$, there exists $C$ such that we have for all
$g \in \Sigma^n(\R)$ with $n \in \N^*$, 
 \[  \| a^{\beta} \pi_h \partial_{x}^{\alpha} g \|_{\ell^2(h\Z)} \leq C \| g \|_{\Sigma^n(\R)}.  \] 
 \end{lemma}
 \begin{proof}
 We first recall a classical lemma from functional analysis, whose complete proof can be found in \cite{helffer1984} or \cite{abdallah2008}: for all $\alpha$, $\beta \in \N$ such that $\alpha+\beta\leq n$, we have
\[\| x^\alpha  \partial_{x}^{\beta} g \|_{L^2(\R)} \leq C  \| g \|_{\Sigma^n(\R)} \]
for some constant $C$ independent of $g$. 
 It naturally follows that if $\alpha+\beta\leq n-1$, then
 \begin{equation}
 \label{eq2} \| x^{\beta} \partial_{x}^{\alpha} g \|_{H^1(\R)} \leq C  \| g \|_{\Sigma^n(\R)}.  
 \end{equation}
 as 
  \[ \partial_x \left( x^{\beta} \partial_{x}^{\alpha} g \right) = \beta x^{\beta-1} \partial_{x}^{\alpha} g +x^{\beta} \partial_{x}^{\alpha+1} g.  \]
  Combining the fact that 
 \[ \pi_h \left( x^{\beta} \partial_{x}^{\alpha} g  \right) = a^{\beta} \pi_h \partial_{x}^{\alpha} g  \]
 with equation \eqref{eq2} and Lemma \ref{lemma_L2_H1}, we then get the result.
 \end{proof}

\begin{proofof}{Theorem}{\ref{theorem_kravchuk_oscillator}}
 Let $g \in \Sigma^5(\R)$ and $a \in h \Z$, we have
 \[ (\pi_h H g)(a)= -g''(a) +a^2 g(a), \]
  ($g'' \in H^3(\R)$ well admits a continuous representative), and 
 \begin{multline*}
  (H_h \pi_h g)(a)= - \frac{1}{h^2} \sqrt{(1+ah+h^2)(1-ah)} g(a+h) \\
   - \frac{1}{h^2} \sqrt{(1-ah+h^2)(1+ah)} g(a-h) + \left( 1+\frac{2}{h^2} \right) g(a)   
  \end{multline*}
 if $a \in A_h=h\Z \cap \left[ -\frac{1}{h},\frac{1}{h} \right]$, and $H_h \pi_h g(a)=0$ elsewhere. We first compute that
  \[ (1\mp ah+h^2)(1\pm ah)=1+h^2(1-a^2\pm ah),   \]
  and we denote 
  \[ R_h^{\pm}(a):= 1+\frac{h^2}{2} \left( 1-a^2 \pm ah \right) - \sqrt{1+h^2 \left( 1-a^2 \pm ah \right)}  \]
  for all $a \in A_h$, and $R_h^{\pm}(a):=0$ elsewhere.
  We have 
  \begin{align*}
  ( H_h \pi_h g - \pi_h H )(a) &= 
   g''(a) - \Delta_h g(a) - (a^2-1) g(a) \\
   & + \frac{1}{h^2}R_h^{-}(a) g(a + h)  + \frac{1}{h^2}R_h^{+}(a) g(a - h) \\
&    -(1-a^2 + ah) g(a-h) - (1-a^2 - ah)   g(a+h ), 
  \end{align*} 
  where
  \[ \Delta_h u(a) = \frac{ u(a+h) + u(a-h) - 2 u(a) }{h^2}.    \]
%  so from Taylor integral formula we can write that
%  \begin{multline*}
%    \sqrt{1+h^2(1-a^2 \pm ah)} = 1 + \frac{h^2}{2}(1-a^2 \pm ah) \\
%     - \frac{1}{4} \int_0^{h^2(1-a^2 \pm ah)} \frac{h^2(1-a^2 \pm ah) - t}{(1+t)^{\frac{3}{2}}} dt, 
%    \end{multline*}
%  so 
% \begin{align*}
% 	H_h \pi_h g(a) & =  \Delta_h g(a) + (1-a^2) \frac{g(a+h)+g(a-h)}{2} \\
% 	& + \frac{1}{2}ah \left( g(a-h)-g(a+h) \right)\\
% 	& - \frac{g(a+h)}{4h^2} \int_0^{h^2(1-a^2-ah)} \frac{h^2(1-a^2-ah) - t}{(1+t)^{\frac{3}{2}}} dt \\
% 	& - \frac{g(a-h)}{4h^2} \int_0^{h^2(1-a^2+ah)} \frac{h^2(1-a^2+ah) - t}{(1+t)^{\frac{3}{2}}} dt
% \end{align*}
% for all $a \in A_h$. 
 Hence, by a triangular inequality, we get that
\[ \| \pi_h H g - H_h \pi_h g \|_{\ell^2(h\Z)} \leq \mathcal{S}_1+\mathcal{S}_2+\mathcal{S}_3+\mathcal{S}_4+\mathcal{S}_5, \]
with
\begin{align*} &\mathcal{S}_1^2 = h \sum_{a \in h\Z} \left|g''(a)- \Delta_h g(a) \textbf{1}_{A_h}(a) \right|^2, \\
& \mathcal{S}_2^2=h \sum_{a \in h\Z} \left|(a^2-1) \left( g(a) - \frac{g(a+h)+g(a-h)}{2}\textbf{1}_{A_h}(a) \right) \right|^2,  \\
& \mathcal{S}_3^2 = h \sum_{a \in A_h} \left| ah \frac{g(a-h)-g(a+h)}{2} \right|^2 	, \\
& \mathcal{S}_4^2 = h \sum_{a \in A_h} \left| \frac{g(a+h)}{h^2} R_h^{-}(a) \right|^2  \quad\mbox{and} \quad \mathcal{S}_5^2 = h \sum_{a \in A_h} \left| \frac{g(a-h)}{h^2} R_h^{+}(a) \right|^2.  
\end{align*}

We first look at $\mathcal{S}_1$, and we split the sum over $A_h$ and $h\Z \backslash A_h$, so that
\[ \mathcal{S}_1^2 = h \sum_{a \in A_h} \left|g''(a)- \Delta_h g(a) \right|^2 + h \sum_{a \notin A_h} \left|g''(a) \right|^2. \]
For all $a \in A_h$, as $g \in H^5(\R) \hookrightarrow \mathcal{C}^4(\R)$ by Taylor formula we get that 
\[g(a+h)=g(a)+h g'(a) + \frac{h^2}{2} g''(a) + \frac{h^3}{6} g^{(3)}(a) + \int_a^{a+h} \frac{(a+h-s)^3}{6} g^{(4)}(s) \dd s\]
and
\[g(a-h)=g(a)-h g'(a) + \frac{h^2}{2} g''(a) - \frac{h^3}{6} g^{(3)}(a) - \int_{a-h}^{a} \frac{(a-h-s)^3}{6} g^{(4)}(s) \dd s,\]
hence by direct cancellations
\begin{multline*}
 g''(a)- \frac{g(a+h) + g(a-h)-2 g(a)}{h^2} \\
 = - \frac{1}{6h^2} \left( \int_a^{a+h} (a+h-s)^3 g^{(4)}(s) \dd s - \int_{a-h}^{a} (a-h-s)^3 g^{(4)}(s) \dd s \right).
\end{multline*}
By the standard inequality $|a-b|^2 \leq 2 (|a|^2+|b|^2)$, we infer that
\begin{multline*}
 \left| g''(a)- \Delta_h g(a) \right|^2 \\
 =  \frac{1}{18 h^4} \left( \left| \int_a^{a+h} (a+h-s)^3 g^{(4)}(s) \dd s \right|^2 + \left| \int_{a-h}^{a} (a-h-s)^3 g^{(4)}(s) \dd s \right|^2 \right),
\end{multline*}
so by Cauchy-Schwarz inequality we get, for instance for the first integral, that
\begin{align*}
\left| \int_a^{a+h} (a+h-s)^3 g^{(4)}(s) \dd s \right|^2 & \leq \left( \int_a^{a+h} (a+h-s)^6  \right) \left( \int_a^{a+h} |g^{(4)}(s)|^2 \dd s   \right) \\
& = \frac{h^7}{7} \int_a^{a+h} |g^{(4)}(s)|^2 \dd s,
\end{align*}
and the same way,
\[ \left| \int_a^{a+h} (a-h-s)^3 g^{(4)}(s) \dd s \right|^2 \leq  \frac{h^7}{7} \int_{a-h}^{a} |g^{(4)}(s)|^2 \dd s \]
so finally, as $A_h \subset h\Z$,
\begin{equation} \label{first_estimate_S1}
h \sum_{a \in A_h} \left| g''(a)- \Delta_h g(a) \right|^2 \leq h \sum_{a \in h\Z} \frac{1}{18h^4} \frac{ h^7}{7}  \int_{a-h}^{a+h} |g^{(4)}(s)|^2 \dd s \leq \frac{h^4}{63} \Norm{g^{(4)}}{L^2(\R)}^2. 
\end{equation}
%\begin{align}\nonumber 
% |g''(a)- \Delta_h g(a)|^2 &\leq \frac{1}{h^4 4!} \left|  \int_{a}^{a+h} g^{(4)}(s) ( a + h - s)^3
% +  \int_{a-h}^{a} g^{(4)}(s) ( a - s)^3\right|^2\\ \label{raslbol}
% &\leq \frac{C}{h^4}  \left(\int_{0}^{a+h} |g^{(4)}(s)|^2 \dd s \right)\left(\int_{0}^{a+h}  ( a + h - s)^6 \dd s \right)\\\nonumber 
% & + \frac{C}{h^4}  \left(\int_{a-h}^{0} |g^{(4)}(s)|^2 \dd s \right)\left(\int_{a-h}^{0}  ( a  - s)^6 \dd s \right)\\\nonumber 
% &\leq C h^3  \int_{a-h}^{a+h} |g^{(4)}(s)|^2 \dd s, 
%\end{align}
%and thus we obtain 
%\[ \left( h \sum_{a \in A_h} |g''(a)- \Delta_h g(a)|^2 \right)^{\frac{1}{2}} \leq C h^2 \Norm{g^{(4)}}{L^2(\R)}.  \]
On the other hand, as $|a| \geq 1/h$,
\begin{align*} \sum_{a \notin A_h} |g''(a)|^2 &= \sum_{a \notin A_h} \frac{1}{|a|^4} |a|^4 | g''(a)|^2 \leq h^4 \sum_{a \notin A_h} |a|^4 |g''(a)|^2 \leq C h^3 \| a^2 g''(\cdot) \|_{\ell^2(h\Z)}^2, 
\end{align*}
using the fact that $h\Z \backslash A_h \subset h\Z$, so
\begin{equation} \label{second_estimate_S1}
 \left( h \sum_{a \notin A_h} |g''(a)|^2 \right) \leq C h^4 \| g \|_{\Sigma^5(\R)}^2 
 \end{equation}
using Lemma \ref{weighted_discrete_sobolev_prop}, as in particular
$  \|a^2 g''(\cdot) \|_{\ell^2(h\Z)}^2 \leq C \|g \|_{\Sigma^5(\R)}^2$.  
Finally, combining estimates \eqref{first_estimate_S1} and \eqref{second_estimate_S1}, as $\|g^{(4)} \|_{L^2(\R)} \leq C \|g \|_{\Sigma^5(\R)}$, we get that
\[ \mathcal{S}_1 \leq C h^2  \|g \|_{\Sigma^5(\R)}. \]

In the same vein, we see that
\[ \mathcal{S}_2^2=h \sum_{a \in A_h} \left|(a^2-1) \left( g(a) - \frac{g(a+h)+g(a-h)}{2} \right) \right|^2  + h \sum_{a \notin A_h} |(a^2-1)g(a) |^2. \]
First, for $a \in A_h$, as from Taylor formula
\[
\begin{array}{l}
 g(a+h)=g(a) + h g'(a) + \int_a^{a+h} (a+h-s) g''(s) \dd s \\[1ex]
  g(a-h)=g(a) - h g'(a) - \int_{a-h}^a (a-h-s) g''(s) \dd s,  
  \end{array}
  \]
we have
\[ g(a) - \frac{g(a+h)+g(a-h)}{2} =-\int_a^{a+h} (a+h-s) g''(s) \dd s + \int_{a-h}^a (a-h-s) g''(s) \dd s.   \]
Estimating for instance the first integral, we naturally get that
\begin{align*}
 \left| \int_{a}^{a+h} (a+h-s) g''(s) \dd s \right|^2 & \leq \left( \int_a^{a+h} (a+h-s)^2 \dd s  \right) \left( \int_a^{a+h} |g''(s)|^2 \dd s \right) \\
 & \leq  \frac{h^3}{3}  \int_a^{a+h} |g''(s)|^2 \dd s  
\end{align*}
from Cauchy-Schwarz inequality, so
\[ h \sum_{a \in A_h} |a^2-1|^2 \left| g(a) - \frac{g(a+h)+g(a-h)}{2} \right|^2 \leq h^4 \sum_{a \in A_h} \langle a \rangle^2 \int_{a-h}^{a+h} |g''(s)|^2 \dd s  \]
from the standard inequality $|a^2 - 1|^2 \leq 2 (1+|a|^2)=2 \langle a \rangle^2$. Let $a \geq h$. If $s \in \left[a, a+h \right]$, we naturally get that $\langle a \rangle^2 \leq \langle s \rangle^2$, so
\[ \langle a \rangle^2 \int_{a}^{a+h} |g''(s)|^2 \dd s \leq  \int_{a}^{a+h} \langle s \rangle^2 |g''(s)|^2 \dd s. \]
If $s \in \left[a-h, a \right]$, we can write that
\[ 1 \leq 1+(a-h)^2 \leq 1+s^2,    \]
so
\[ 1 +a^2 \leq 1+s^2 +2ah-h^2 \leq 3+s^2 \leq 3 \langle s \rangle^2 \]
as $a\leq 1/h$ and $0< h \leq 1$, so
\[ \langle a \rangle^2 \int_{a-h}^{a} |g''(s)|^2 \dd s \leq  3 \int_{a-h}^{a} \langle s \rangle^2 |g''(s)|^2 \dd s. \]
As this bound is obvious for $a=0$ and entirely symmetrical for $a\leq -h$, we finally get that 
\begin{multline*}  
\left( h \sum_{a \in A_h} \left|(a^2-1) \left( g(a) - \frac{g(a+h)+g(a-h)}{2} \right) \right|^2 \right)^{\frac{1}{2}} \\
\leq C h^2 \left(\sum_{a \in A_h} \int_{a-h}^{a+h} \langle s \rangle^2 |g''(s)|^2 \dd s \right)^{\frac{1}{2}} \leq C h^2 \| g \|_{\Sigma^4(\R)} \leq C h^2 \| g \|_{\Sigma^5(\R)}.  
\end{multline*}
Now, for $a \notin A_h$, we write that
\[ |a^2 g(a) - g(a) |^2 \leq 2 \left( |a^2g(a)|^2 +|g(a)|^2 \right),  \]
so as for $\mathcal{S}_1$ we have
\begin{align*}
h \sum_{a\notin A_h} |a^2g(a)|^2 & = h \sum_{a\notin A_h} \frac{1}{a^4} |a|^4 |a^2 g(a)|^2 \leq h^5 \sum_{a\notin A_h} |a^4 g(a)|^2 \\
								 & \leq h^4 \| a^4 g(\cdot) \|_{\ell^2(h\Z)}^2 \leq h^4 \|g \|^2_{\Sigma^5(\R},
\end{align*}
and
\[ h \sum_{a\notin A_h} |g(a)|^2 = h \sum_{a\notin A_h} \frac{1}{a^4}|a^2 g(a)|^2 \leq C h^4 \| g \|_{\Sigma^5(\R)}  \]
from Lemma \ref{weighted_discrete_sobolev_prop}. So finally we well have
\[ \mathcal{S}_2 \leq C h^2  \|g \|_{\Sigma^5(\R)}. \]

%so by using again Taylor expansions
%\begin{multline*}  
%\left( h \sum_{a \in A_h} \left|(a^2-1) \left( g(a) - \frac{g(a+h)+g(a-h)}{2} \right) \right|^2 \right)^{\frac{1}{2}} \\
%\lesssim h^2 \left(\sum_{a \in A_h} \int_{a-h}^{a+h} |x^2 g''(x)|^2 \right)^{\frac{1}{2}} \leq C h^2 \| g \|_{\Sigma^4(\R)},    
%\end{multline*}
%and
%\[ \sum_{a \notin A_h} |(a^2-1)g(a) |^2 =  \sum_{a \notin A_h} \frac{1}{a^3}|a^\frac32(a^2-1)g(a) |^2 \leq C h^3 \sum_{a \notin A_h} |\langle a\rangle^\frac{5}{2} g(a)|^2,   \]
%so
%\[ \left( h \sum_{a \notin A_h} |(a^2-1)g(a) |^2 \right)^{\frac{1}{2}} \leq C h^2 \| g \|_{\Sigma^4(\R)}.  \]

Now, looking at $\mathcal{S}_3$, we use the fact that, from the same Taylor formula and Cauchy-Schwarz inequality as before, 
we have
\begin{multline*}
 ah \frac{g(a-h)-g(a+h)}{2} \\
 = -ah^2g'(a) - \frac{ah}{2} \left( \int_a^{a+h} (a+h-s) g''(s) \dd s + \int_{a-h}^a (a-h-s) g''(s) \dd s   \right),  
 \end{multline*}
so
\begin{align*} 
\left| ah \frac{g(a-h)-g(a+h)}{2} + h^2 a g'(a) \right|^2 &
\leq C h^5 |a|^2 \int_{a-h}^{a+h} |g''(s)|^2 \dd s   \\
& \leq 
 C h^3 \int_{a-h}^{a+h} |g''(s)|^2 \dd s  
\end{align*}
as $|a| \leq 1/h$ for all $a \in A_h$, so
\begin{align*}
   \mathcal{S}_3 = &\left( h \sum_{a \in A_h} \left| ah \frac{g(a-h)-g(a+h)}{2} \right|^2 \right)^{\frac{1}{2}} \\
   &\leq     \left( h \sum_{a \in A_h} | ah^2 g'(a)|^2 \right)^{\frac{1}{2}}    + C h^2 \|  g'' \|_{L^{2}(\R)} \leq C h^2  \| g \|_{\Sigma^5(\R)}. 
  \end{align*}

It now remains to bound $S_4$ and $S_5$ in terms of $\| g \|_{\Sigma^5(\R)}$. As these two sums are symmetric, we will only show how to control $S_4$ in the following. Let us first analyze the behavior of the function 
\[
R_h^{-}(a)=1+h^2(1-a^2-ah)/2-\sqrt{1+h^2(1-a^2- ah)} = f( h^2( 1 - a^2 - ah) ) 
\]
where $f(x) = 1 + \frac12 x - \sqrt{1 + x}$. We have that $f''(x) = \frac14 ( 1 + x)^{-\frac32}$ which is uniformly bounded on $\mathbb{R}$, and moreover $f(0) = f'(0) = 0$, from which we deduce that $|f(x) | \leq C x^2$. Hence 
\begin{align*}
h \sum_{a \in A_h} \left| \frac{g(a+h)}{h^2} R_h^{-}(a) \right|^2 &\leq 
h \sum_{a \in A_h} \left| g(a+h) h^2 ( 1 - a^2 - ah) \right|^2\\
&\leq h^4 h \sum_{a \in h \Z} \left| g(a) \langle a \rangle^2 \right|^2 \leq 
C h^4 \| g \|_{\Sigma^5(\R)}^2, 
\end{align*}
which ends the proof of Theorem \ref{theorem_kravchuk_oscillator}. 
 \end{proofof}

\section{Convergence of Kravchuk functions} \label{kravchuk_functions_section}

The aim of this section is to prove the second part of Theorem \ref{theorem_kravchuk_oscillator}. 

\subsection{Convergence of the binomial law} 
We first prove \eqref{convrho}.
We define the projection of the Gaussian function $x \mapsto e^{-x^2}$ on the grid $h\Z$ by
\[ \rho : a \in h \Z \mapsto e^{-a^2}.   \]
%It is quite straightforward that both the functions $\rho_h$ and $\rho$ belong to the discrete space $L^2(h \Z)$. We now state the first proposition of this section, which states that the renormalized binomial distribution $\rho_h$ well approximates the discrete Gaussian $\rho$ on $L^2(h \Z)$ in the limit $h \rightarrow 0$. More precisely:
%
%
%\begin{proposition} \label{binomial_to_gaussian}
%For all $\delta \in (0,1)$ and $\sigma \geq 0$, there exists $C$ and $h_0$ such that for all $h \leq h_0$, 
%\[ \left\| \langle a \rangle^{\sigma} \left(\rho_h(a) - \frac{1}{\sqrt{\pi}} e^{- a^2}\right)\right\|_{L^2(h \Z)} \leq C h^{2-\delta}.   \]
%\end{proposition}
%\begin{proof}
By definition we have, using \eqref{eq:rhoha} and the definition \eqref{binomial} with $N = \frac{2}{h^2}$,
\[   \rho_h(a)=\frac{1}{h}\Pi_N \left( \tau_h^{-1}(a)  \right)= \frac{1}{h 4^{1/h^2}} \frac{ \Gamma \left( 1 + \frac{2}{h^2} \right)  }{\Gamma \left( 1 + \frac{1}{h^2} + \frac{a}{h} \right) \Gamma \left( 1 +\frac{1}{h^2} - \frac{a}{h}  \right)},  \]
where $\Gamma$ denotes the usual Gamma function such that $\Gamma(n+1) = n!$ for integers. The Stirling asymptotics (see for instance \cite[6.1.42, p.257]{abramowitz64}) yields
\begin{equation}\label{restestirling}
 \log \Gamma (z)=  \left( z - \frac{1}{2} \right) \log(z) - z + \frac{1}{2} \log( 2 \pi) + R_0(z)  
\quad{\mbox{with}}\quad
 |R_0(z) | \leq \frac{c_0}{\langle z\rangle} .    
 \end{equation}
 for some constant $c_0$. This asymptotic yields in particular the Stirling formula
 \[
 n! = \Gamma(1 + n) \simeq\sqrt{2\pi n}\,  n^n e^{-n}. 
 \]
This shows that 
\begin{align*}
&\log  \rho_h(a)  = -  \log h  - \frac{2}{h^2}\log 2 + \log \Gamma \left( 1 + \frac{2}{h^2} \right)\\
&  
- \log \Gamma \left( 1 + \frac{1}{h^2} + \frac{a}{h} \right) - \log  \Gamma \left( 1 +\frac{1}{h^2} - \frac{a}{h}  \right)\\
&= -  \log h  - \frac{2}{h^2}\log 2  + \left( \frac{2}{h^2}   + \frac{1}{2} \right) \log \left(1 + \frac{2}{h^2} \right) - \left(1 + \frac{2}{h^2} \right)  + \frac{1}{2} \log( 2 \pi) \\
& -  
  \left(  \frac{1}{h^2} + \frac{a}{h}   + \frac{1}{2} \right) \log \left(1 + \frac{1}{h^2} + \frac{a}{h} \right) + \left( 1 + \frac{1}{h^2} + \frac{a}{h} \right)  - \frac{1}{2} \log( 2 \pi) \\
&- \left(  \frac{1}{h^2} - \frac{a}{h} + \frac{1}{2}  \right) \log \left( 1 +\frac{1}{h^2} - \frac{a}{h}  \right) +  \left(1 +\frac{1}{h^2} - \frac{a}{h} \right)   - \frac{1}{2} \log( 2 \pi) \\
& + R_0 \left(1 + \frac{2}{h^2} \right)- R_0 \left(1 + \frac{1}{h^2} + \frac{a}{h} \right)- R_0 \left( 1 +\frac{1}{h^2} - \frac{a}{h}  \right),
\end{align*}
so we can write that
\begin{align*}
\log  \rho_h(a)
&= - \frac{1}{2} \log( 2 \pi)  -  \log h  - \frac{2}{h^2}\log 2 \\
& - \left(1 + \frac{2}{h^2} \right)+ \left( 1 + \frac{1}{h^2} + \frac{a}{h} \right)  +  \left(1 +\frac{1}{h^2} - \frac{a}{h} \right) \\
& + \left( \frac{2}{h^2}   + \frac{1}{2} \right) \left[\log 2 - 2 \log h  + \log \left(1 + \frac{h^2}{2} \right)\right]\\
&-\left(  \frac{1}{h^2} + \frac{a}{h} + \frac{1}{2} \right)\left[- 2 \log h+\log(1  + ah + h^2)\right]\\
&-\left(  \frac{1}{h^2} - \frac{a}{h} + \frac{1}{2} \right)\left[- 2 \log h+\log(1  - ah + h^2)\right]\\
& + R_0 \left(1 + \frac{2}{h^2} \right)- R_0 \left(1 + \frac{1}{h^2} + \frac{a}{h} \right)- R_0 \left( 1 +\frac{1}{h^2} - \frac{a}{h} \right),
\end{align*}
and thus
\begin{align}
\log  \rho_h(a)
&= 1 - \frac{1}{2} \log( \pi)    \label{stooop}  + \left( \frac{2}{h^2}   + \frac{1}{2} \right) \log \left(1 + \frac{h^2}{2} \right) \\
&-\left(  \frac{1}{h^2} + \frac{a}{h} + \frac{1}{2} \right)\log(1  + ah + h^2)\nonumber  -\left(  \frac{1}{h^2} - \frac{a}{h} + \frac{1}{2} \right)\log(1 - ah + h^2)\nonumber  \\
& + R_0 \left(1 + \frac{2}{h^2} \right)- R_0 \left(1 + \frac{1}{h^2} + \frac{a}{h} \right)- R_0 \left( 1 +\frac{1}{h^2} - \frac{a}{h} \right).
\end{align}
Let us estimate this term first in the regime $|a| \leq h^{- \delta}$ with $\delta \in (0,1)$. In this case, we have 
\[
\frac{1}{h^2} \pm \frac{a}{h} = \frac{1}{h^2}(1 + \mathcal{O}(h^{1 -\delta})) \geq c h^{-2}, 
\]
for $h$ small enough. This shows that in the previous expression, the terms with $R_0$ can be estimated with \eqref{restestirling} and are of order $\mathcal{O}(h^2)$. We also have $|h^2 \pm ah| \leq C h^2$
and thus we can expand the $\log$ terms by using 
$\log (1 + x) = x - \frac{x^2}{2} + \mathcal{O}(|x|^3)$ and we obtain
\begin{align*}
\log  \rho_h(a)
&= 1 - \frac{1}{2} \log( \pi)   \\
& + \left( \frac{2}{h^2}   + \frac{1}{2} \right) \left( \frac{h^2}{2} - \frac{h^4}{4} + \mathcal{O}(h^6)\right)\\
&-\left(  \frac{1}{h^2} + \frac{a}{h} + \frac{1}{2} \right)\left( ah + h^2 - \frac12 ( ah + h^2)^2 + \mathcal{O}(h^6)\right) \\
&-\left(  \frac{1}{h^2} - \frac{a}{h} + \frac{1}{2} \right)\left( -ah + h^2 - \frac12 ( -ah + h^2)^2 + \mathcal{O}(h^6)\right)\\
& +\mathcal{O}( h^2)
\end{align*}
and thus
\begin{align*}
\log  \rho_h(a)
&= 1 - \frac{1}{2} \log( \pi)   + 1     +\mathcal{O}( h^2)\\
&-\left(  \frac{1}{h^2} + \frac{a}{h} + \frac{1}{2} \right)\left( ah + h^2 - \frac12 a^2 h^2 - ah^3   \right) \\
&-\left(  \frac{1}{h^2} - \frac{a}{h} + \frac{1}{2} \right)\left( -ah + h^2 - \frac12 a^2 h^2 + ah^3 \right),
\end{align*}
yielding
\begin{align*}
\log  \rho_h(a)
&= 2 - \frac{1}{2} \log( \pi)  +\mathcal{O}( h^2)\\
&-\left(  \frac{a}{h} + a^2 + ah+  1 + ah + \frac{h^2}{2} 
 - \frac12 a^2  - \frac12 a^3 h   - \frac14 a^2 h^2   - ah  - a^2 h^2 \right) \\
&-\left(  -\frac{a}{h} + a^2 - ah+  1 - ah + \frac{h^2}{2} 
 - \frac12 a^2  + \frac12 a^3 h   - \frac14 a^2 h^2   + ah  - a^2 h^2 \right) 
\end{align*}
or 
\begin{align*}
\log  \rho_h(a)
&= 2 - \frac{1}{2} \log( \pi)  -2 \left(  a^2 +  1 
 - \frac12 a^2      - \frac{5}{4}a^2 h^2 \right)  +\mathcal{O}( h^2)\\
& = - \frac{1}{2} \log( \pi)  - a^2   +\mathcal{O}( h^{2 - \delta}).
\end{align*}
We thus obtain that
\[ \left| \rho_h (a) - \frac{1}{\sqrt{\pi}} e^{-a^2} \right| \leq C h^{2-\delta} e^{-a^2} , \quad a \in  h\Z \cap [- h^{-\delta},h^{-\delta}],
\]
from which we deduce that
\begin{align*}
h \sum_{a \in h\Z, |a| \leq h^{-\delta}} \langle a \rangle^{2 \sigma} \left| \rho_h (a) - \frac{1}{\sqrt{\pi}} e^{-a^2} \right|^2
&\leq C h^{4 - 2\delta} h \sum_{a \in h\Z, |a| \leq h^{-\delta}} \langle a \rangle^{2 \sigma} e^{- a^2}\\
& \leq C h^{4 - 2\delta} h \sum_{a \in h\Z} \langle a \rangle^{2 \sigma} e^{- a^2} \leq C_\sigma h^{4 - 2\delta}. 
\end{align*}
Now from \eqref{stooop}, we have that for $|a| \geq h^{-\delta}$ and $\beta > 0$, by using the fact that the $R_0$ terms are uniformly bounded in $h$, 
\begin{align*}
\log e^{\beta a^2}  \rho_h(a)
&= \beta a^2     + \left( \frac{2}{h^2}   + \frac{1}{2} \right) \log \left(1 + \frac{h^2}{2} \right)  + \mathcal{O}(1)\\
&-\left(  \frac{1}{h^2} + \frac{a}{h} + \frac{1}{2} \right)\log(1  + ah + h^2) -\left(  \frac{1}{h^2} - \frac{a}{h} + \frac{1}{2} \right)\log(1 - ah + h^2). 
\end{align*}
In the case $a \in [h^{-\delta},h^{-1}]$, we write $a = h^{-1}(1 - b)$, $b \in [0,h^{1 - \delta}]$ and obtain 
\begin{align*}
\log  e^{\beta a^2} \rho_h(a)
&= \frac{\beta}{h^2}(1 - b)^2 -\left( \frac{2 - b}{h^2} + \frac{1}{2} \right)\left( \log 2 + \log \left(1 - \frac{b}{2} + \frac{h^2}{2} \right) \right) \\
&-\left(  \frac{b}{h^2} + \frac{1}{2} \right)\log(b + h^2) + \mathcal{O}(1)\\
&= \frac{\beta}{h^2}(1- b)^2 - \frac{1}{h^2}\left(2\log 2 + \mathcal{O}\left(h^{1 - \frac{\delta}{2}} \right)\right)\\
&=  \frac{1}{h^2}\left(\beta - 2\log 2 + \mathcal{O}(h^{1 - \delta}\right) \leq 0
\end{align*}
for $\beta = 1 < 2 \log 2$ and $h$ small enough. 
By symmetry, 
we deduce that 
\[
\forall\, a \in [h^{-\delta},h^{-1}], \quad |\rho_h(a)| \leq e^{-  a^2} 
\]
Hence 
\begin{align*}
h \sum_{a \in h\Z \cap [h^{-\delta},h^{-1}] } \langle a \rangle^{2\sigma} |\rho_h(a)|^2 &\leq 
h \sum_{a \in h\Z \cap [h^{-\delta},h^{-1}] } \langle a \rangle^{2\sigma} e^{- 2  a^2}\\
&\leq e^{- \beta h^{-\delta}} h \sum_{a \in h\Z \cap [h^{-\delta},h^{-1}] } \langle a \rangle^{2\sigma} e^{-  a^2} \leq C_\sigma e^{- \beta h^{-\delta}} h  = \mathcal{O}(h^2). 
\end{align*}
As the same bounds holds for the Gaussian $e^{- a^2}$, we obtain the result. 

\subsection{Convergence of the Kravchuk functions}
%We now denote $(H_n)_{n \in \N}$ the Hermite polynomials and $(\psi_n)_n$ the associated Hermite functions (see Appendix \ref{harmonic_oscilator_appendix} for the precise definition and some reminders about these functions). We see in the previous proposition that the recurrence relation satisfied by the functions $(\tilde{k}_{n,h})_n$ is quite similar to the one satisfied by the Hermite polynomials. This allows us to show the convergence of the renormalized Kravchuk functions to Hermite functions in the following.
%
We now prove \eqref{convphi}. 
Recall that the Hermite polynomials are defined by the relation 
\begin{equation}
\label{eqHermite2}
H_{n+1}(x) = 2x H_n(x) - 2n H_{n-1}(x), \quad H_0(x) = 1, 
\end{equation}
while the scaled Kravchuk polynomials are defined by the relation 
\[ K_{n+1,h}(x) =2x K_{n,h}(x) - 2n \left( 1 - h^2 \Big(\frac{n-1}{2}\Big) \right) K_{n-1,h}(x), \quad K_{0,h}(x) = 1. 
\]
The Hermite functions are then defined by the formula 
\[
 \psi_n(x) := \frac{1}{ \pi^{\frac14} 2^{\frac{n}{2}} \sqrt{n!} } e^{-\frac{x^2}{2}} H_n(x),
\]
and let us recall that Cram\'er's inequality states that 
\[
\forall\, x \in \R, \quad \forall\, n \geq 0, \quad |\psi_n(x)| \leq \pi^{-\frac14},
\]
which shows in particular that 
\begin{equation}
\label{Cramer}
\forall\, x \in \R, \quad \forall\, n \geq 0, \quad |H_n(x)| \leq e^{\frac{x^2}{2}}2^{\frac{n}{2}} \sqrt{n!} .
\end{equation}

\begin{lemma}
For any given $N$ and $h = \sqrt{2}N^{-\frac12}$, we denote by $(k_{n,h})_n$ the Kravchuk polynomials given by the relation \eqref{eqKravchuk}. We then have the following bounds: there exists $C$ such that for all $N$ and $h = \sqrt{2}N^{-\frac12}$, for all $\delta \in (0,1)$,
\begin{equation}
\label{HK1}
\forall\, n \leq \frac13 \delta |\log h|, 
 \quad  \forall\, x \in \R, \quad |H_{n}(x) - k_{n,h}(x)| \leq C h^{2 - \delta} e^{\frac{x^2}{2}} 2^{\frac{n}{2}}\sqrt{n!}
\end{equation}
and
\begin{equation}
\label{boundK}
\forall\, n \leq \frac13 \delta |\log h|, 
 \quad  \forall\, x \in \R, \quad |k_{n,h}(x)| \leq C h^{- \delta} e^{\frac{x^2}{2}} 2^{\frac{n}{2}}\sqrt{n!}.
 \end{equation}
\end{lemma}
\begin{proof}
We first remark that \eqref{HK1} and \eqref{Cramer} imply \eqref{boundK}, so we only need to prove \eqref{HK1}. Let us define 
\[
H(x,t) = \sum_{n \geq 0}
 \frac{t^n}{n!} H_n(x) \quad \mbox{and}\quad K_h(x,t) = \sum_{n \geq 0}
 \frac{t^n}{n!} k_{n,h}(x).
 \]
Multiplying \eqref{eqHermite2} by $\frac{t^n}{n!}$ we can prove that the function $H(t,x)$ satisfies the equation $\partial_t H(t,x) = (2x - 2t)H(t,x)$ and we obtain the classical relation 
\[
H(x,t) = e^{2xt - t^2},
\]
valid for all $x \in \R$ and $t \in \C$.  Moreover, relation \eqref{eqKravchuk} yields 
 \[
 \sum_{n \geq 0} \frac{t^n}{n!} k_{n+1,h}(x) = 2x k_h(x,t) - 2 \sum_{n\geq 0} n \frac{t^n}{n!} k_{n-1,h}(x)  + h^2 \sum_{n\geq 0} n (n-1) \frac{t^n}{n!} k_{n-1,h}(x)
 \]
 or 
 \begin{align*}
 \partial_t K_{h}(x,t) &=  2xK_h(x,t)   - 2 tK_h(x,t)  + h^2  t^{2}\sum_{n\geq 2}  \frac{t^{n-2}}{(n-2)!} k_{n-1,h}(x)
 \\
 &=  2xK_h(x,t)   - 2 tK_h(x,t)  + h^2  t^{2} \partial_{t} K_{h}(x,t), \end{align*}
and we find 
\[
(1 - h^2 t^2) \partial_{t} K_h(x,t)  = (2x   - 2 t)K_h(x,t) 
\]
from which we deduce that
\begin{align*}
K_h(x,t) &= \exp\left(\int_0^t \frac{2x - 2s}{1 - h^2 s^2} \dd s\right) \\
&= \exp\left(2x \frac{1}{h}\int_0^{ht} \frac{1}{1 -  s^2} \dd s\right) \exp\left(\frac{1}{h^2}\int_0^{ht} \frac{-2s}{1 - s^2} \dd s\right) \\
&= \exp\left(2x \frac{1}{h}\mathrm{artanh} (ht)\right) 
 \exp\left(\frac{1}{h^2}\log( 1 - (ht)^2) \right)\\
  &= \left(\frac{1 + ht}{1 - ht}\right)^{\frac{x}{h}} \left( 1 - h^2 t^2\right)^{\frac{1}{h^2}}. 
\end{align*}
\begin{remark}
This formula shows in particular that for $n > N$, the Kravchuk polynomials vanish. 
\end{remark}
This shows that for $t \in \C$, $|t| \leq R$ with $R > 1$,  we have for $h$ small enough  
\begin{align*}
|H(x,t) - K_{h}(x,t)|&=|e^{2xt - t^2}|| 1 - e^{2x ( \frac{1}{h}\mathrm{artanh} (ht) - 1)}e^{\frac{1}{h^2}\log( 1 - (ht)^2) + t^2}|\\
&= |e^{2xt - t^2}|| 1 - e^{2x ( \mathcal{O}(h^2 |t|^3 ) + \mathcal{O}(h^2 |t|^4 )}|\\
&\leq C h^2 ( |x| R^3 + R^4) e^{2|x|R  +R^2 }.
\end{align*}
By Cauchy estimates on the disk $|t| \leq R$, we obtain 
\[
\forall\, n \in \N, \quad 
|H_n(x) - k_{n,h}(x)| \leq  n! R^{-n} C h^2 ( |x| R^3 + R^4) e^{2|x|R  +R^2 } .
\]
We take $R^2 = n$. 
This yields using Stirling expansions 
\begin{align*}
e^{- \frac{x^2}{2}}|H_n(x) - k_{n,h}(x)| &\leq C h^2 ( |x| n^{\frac32} + n^2) n^n  e^{-n } n^{\frac{1}{2}} n^{-\frac{n}{2}} e^{2 |x| \sqrt{n} + n} e^{- \frac{x^2}{2}}\\
&\leq C h^2 ( |x| n^{\frac32} + n^2)  n^{\frac{n}{2}} e^{-n } n^{\frac{1}{2}}  e^{- \frac{1}{2}( x - 2 \sqrt{ n})^2 + 3  n}. 
\end{align*}
Now we remark that
\[
|x|e^{- \frac{1}{2}( x - 2 \sqrt{ n})^2} \leq  |x - 2 \sqrt{n}|e^{- \frac{1}{2}( x - 2 \sqrt{ n})^2} + 2 \sqrt{n}e^{- \frac{1}{2}( x - 2 \sqrt{ n})^2} \leq C \sqrt{n}
\]
for $n \geq 1$. Hence we have 
\begin{align*}
e^{- \frac{x^2}{2}}|H_n(x) -k_{n,h}(x)| &
\leq C h^2 2^{\frac{n}{2}}n^{\frac{n}{2}} e^{- \frac{n}{2}} n^{\frac{1}{4}}
\left( n^{\frac14} n^2 2^{-\frac{n}{2}} e^{- \frac{n}{2}}  e^{3 n} \right)\\
&\leq C h^2 2^{\frac{n}{2}} \sqrt{n!} ( n^{\frac94} e^{ (\frac52 - \frac12\log 2)n }) \\
& \leq C h^2 2^{\frac{n}{2}} \sqrt{n!} e^{\frac52 n } \end{align*}
and this yields the result as $e^{\frac52 n } \leq e^{-\frac56 \delta \log h } \leq h^{- \delta}$. 
\end{proof}
Now let us recall that the Kravchuk functions are given by the formula 
\begin{align*}
 \varphi_{n,h}(a) &:= \frac{1}{d_n h^n 2^n n!} k_{n,h}(a) \sqrt{\rho_h(a)} = \frac{1}{h^n }\sqrt{\frac{(N-n)!}{N! n!}}k_{n,h}(a) \sqrt{\rho_h(a)} \\
& = \frac{1}{h^n \sqrt{n!}}\sqrt{\frac{\Gamma( 1 + \frac{2}{h^2} - n)}{\Gamma( 1 + \frac{2}{h^2})} }k_{n,h}(a) \sqrt{\rho_h(a)}. 
\end{align*}
Using \eqref{restestirling}, we have 
\begin{align*}
\log  \frac{\Gamma( 1 + \frac{2}{h^2} - n)}{\Gamma( 1 + \frac{2}{h^2}) }
&=  \left( \frac{2}{h^2} - n +  \frac{1}{2} \right) \log \left(1+ \frac{2}{h^2} - n\right) -  1 -  \frac{2}{h^2} + n + \frac{1}{2} \log( 2 \pi) \\
&- \left(\frac{2}{h^2} +\frac{1}{2} \right) \log \left(1 + \frac{2}{h^2} \right) + \frac{2}{h^2} + 1 - \frac{1}{2} \log( 2 \pi)\\
&
+ R_0 \left(1 + \frac{2}{h^2} - n \right)  - R_0 \left(1+ \frac{2}{h^2} \right)   - R_0(1 + n)
\end{align*}
and thus
\begin{align*}
\log \frac{\Gamma( 1 + \frac{2}{h^2} - n)}{\Gamma( 1 + \frac{2}{h^2})}
&=  \left( \frac{2}{h^2} - n +  \frac{1}{2} \right) \left( \log 2 - 2 \log h + \log \left(1 - h^2 \frac{(n-1)}{2} \right)\right)    \\
&
+ n - \left(\frac{2}{h^2} +\frac{1}{2} \right) \left( \log 2 - 2 \log h + \log \left(1 + \frac{h^2}{2} \right)\right) \\
&
+ R_0 \left(1 + \frac{2}{h^2} - n \right)  - R_0 \left(1+ \frac{2}{h^2}\right). 
\end{align*}
Now let us assume that $n \leq h^{-\beta}$, we can write that
\[
\log \frac{\Gamma( 1 + \frac{2}{h^2} - n)}{\Gamma( 1 + \frac{2}{h^2})}
=  n - n \log 2+  2 n \log h  -  (n-1) - 1 + \mathcal{O}( h^{2 - \beta}), 
\]
and thus we have 
\[
\sqrt{\frac{\Gamma( 1 + \frac{2}{h^2} - n)}{\Gamma( 1 + \frac{2}{h^2})}} 
= \frac{h^n}{2^{\frac{n}{2}}}( 1 + \mathcal{O}(h^{2 - \beta})).
\]
Hence we have for $n \leq h^{-\beta}$, 
\begin{align*}
\varphi_n(a) &= \frac{h^n}{2^{\frac{n}{2}}}( 1 + \mathcal{O}(h^{2 - \beta})) \frac{1}{\pi^{\frac{1}{4}}h^n \sqrt{n!}}k_{n,h}(a) \sqrt{\sqrt{\pi}\rho_h(a)} \\
&= \frac{1}{2^{\frac{n}{2}}\pi^{\frac{1}{4}} \sqrt{n!}}
k_{n,h}(a)  \sqrt{\sqrt{\pi}\rho_h(a)}( 1 + \mathcal{O}(h^{2 - \beta})) .
\end{align*}
By using the previous bound, we thus obtain the result. 
%\color{black}
%\begin{lemma} \label{lemma_error}
%We denote, for all $a \in h\Z$,
%\[  e_n(a) := \left( \tilde{k}_{n,h}(a)-H_n(a)  \right) e^{-a^2/2}.   \]
%Then for all $n \in \N$,
%\[ \| e_n  \|_{\ell^2(h\Z)} \leq \frac{C_n}{N},\]
%where the sequence $(C_n)_n$ is defined through the 2-terms recurrence
%\[  C_{n+1} = 2 C_n +2nC_{n-1} +2 n(n-1) \| \tilde{k}_{n-1,h} e^{-a^2/2} \|_{\ell^2(h\Z)}, \ \ \ C_0=0, \ \ \ C_1=1.   \]
%\end{lemma}
%
%\begin{remark}
%From Lemma \ref{lemma_error} we also get a uniform bound of $a \mapsto \tilde{k}_{n,h}(a)e^{-a^2/2}$ with respect to $h$ from the previous lemma, namely
%\[ \| k_{n,h} e^{-a^2/2} \|_{\ell^2(h\Z)} \leq  \| e_n  \|_{\ell^2(h\Z)} + \| H_n e^{-a^2/2} \|_{L^2(h \Z)} \leq C_n, \]
%which will be useful in the following.
%\end{remark}
%
%\begin{proposition}
%We have, for all $\delta >0$,
%\[ \| \alpha_{n,h} \varphi_{n,h} - \psi_n \|_{L^2(h \Z)} \leq C_n h^{2-\delta}. \]
%\end{proposition}
%\begin{proof}
%We have the standard estimate
%\begin{multline*}
%    \|  \alpha_{n,h} \varphi_{n,h} - \psi_n \|_{L^2(\R)} \leq \| k_{n,h} \left( \sqrt{\rho_h/h} - e^{-x^2/2}/\pi^{\frac{1}{4}} \right)  \|_{L^2(\R)}\\
%    + \| \left( k_{n,h} - H_n \right) e^{-x^2/2}/\pi^{\frac{1}{4}} \  \|_{L^2(\R)}
%\end{multline*}
%and we conclude by the previous lemma and Proposition \ref{binomial_to_gaussian}, which ends the proof of Theorem \ref{theorem_kravchuk_functions}.
%\end{proof}
%
%
%

\section{Time-dependent scheme} \label{time_dependent}
We consider now the time-dependent discrete Schr\"odinger equation
\begin{equation} \label{discrete_schrodinger}
 i \partial_t \psi = H_h \psi,   
 \end{equation}
with $\psi(0,\cdot)=\psi_0 \in \ell^2(h\Z)$. We define 
\begin{align*}
\label{energie}
E(\psi) &= \frac{1}{h} \sum_{a \in A_h} \left(  - \sqrt{(1-ah+h^2)(1+ah)} \Re (\psi(a) \overline{\psi}(a-h)) + |\psi(a)|^2 \left( 1 + \frac{h^2}{2}\right) \right)\\
&= \langle\psi ,H_h \psi\rangle_{\ell^2(h \Z)}, \nonumber
\end{align*}
where $A_h$ is defined in \eqref{eq:Ah}. 
We first show the conservation of mass and energy property of the discrete harmonic oscillator:
\begin{proposition}
For all $t \in \R$,
\begin{equation} \label{discrete_mass_conservation}
\| \psi(t) \|_{\ell^2(h \Z)}=\| \psi_0 \|_{\ell^2(h \Z)}
\quad \mbox{and} \quad E(\psi(t)) = E(\psi_0). 
\end{equation}
\end{proposition}
\begin{proof}
We multiply \eqref{discrete_schrodinger} by $\overline{\psi}$ and we sum over $h \Z$:
\begin{multline*}  
\sum_{a \in h \Z} i \overline{\psi}(a) \partial_t \psi(a)= \sum_{a \in h \Z} \left( -\frac{1}{h^2} \sqrt{(1+ah+h^2)(1-ah)} \psi(a+h)\overline{\psi}(a) \right. \\
\left. -\frac{1}{h^2} \sqrt{(1-ah+h^2)(1+ah)} \psi(a-h) \overline{\psi}(a) + \left( 1 + \frac{h^2}{2}\right) |\psi(a)|^2 \right)\mathbf{1}_{ -\frac{1}{h} \leq a \leq \frac{1}{h} },  
\end{multline*}
so
\begin{multline*} 
  i \frac12 \frac{\dd}{\dd t} \left( \sum_{a \in h\Z} | \psi(a)|^2 \right) = - \frac{1}{h^2}  \sum_{a \in A_h} \left( \sqrt{(1+ah+h^2)(1-ah)} \psi(a+h)\overline{\psi}(a) \right) \\
   -\sum_{a \in h\Z}  \left( \sqrt{(1-ah+h^2)(1+ah)} \psi(a-h) \overline{\psi}(a) \right) +\left( 1 + \frac{h^2}{2}\right) \sum_{a \in A_h} |\psi(a)|^2 .
\end{multline*}
By a change of variable $a \mapsto a-h$ in the first part of the middle sum of the previous equation, we see that
\begin{multline*}
\sum_{a \in A_h} \left( \sqrt{(1+ah+h^2)(1-ah)} \psi(a+h)\overline{\psi}(a) \right. \\
 \left. + \sqrt{(1-ah+h^2)(1+ah)} \psi(a-h) \overline{\psi}(a) \right) \\
= \sum_{a \in A_h} \left( \psi(a)\overline{\psi}(a-h)+\psi(a-h) \overline{\psi}(a) \right) \sqrt{(1-ah+h^2)(1+ah)}\\
 = 2 \sum_{a \in A_h} \Re \left( \psi(a)\overline{\psi}(a-h)   \right)\sqrt{(1-ah+h^2)(1+ah)}, 
\end{multline*}
so multiplying the previous equation by $h$ and taking the imaginary part, we get the mass conservation.
The second statement is classical using the symmetry of $H_h$. 
\end{proof}
Now we are going to analyze time-dependent scheme \eqref{discrete_schrodinger} and compare it with solutions of the corresponding equation 
$i \partial_t \psi = H \psi$ for the harmonic oscillator. Using \eqref{decompo} recall that if $f \in L^2(\R)$, then we can decompose $f$ on the Hermite-Gauss basis $(\psi_n)_{n \in \N}$ by 
\[ f(x)= \sum_{n \geq0} c_n \psi_n(x)    \]
for all $x \in \R$, with 
\[ c_n=\langle f, \psi_n \rangle_{L^2(\R)}= \int_{\R} f(x) \psi_n(x) \dd x,  \]
so that the solution of $i \partial_t \psi = H \psi$ with initial condition $\psi(0,\cdot)=f$ can be written, for all $t \geq 0$,
\[ \psi(t,\cdot)= e^{-i t H} f = \sum_{n \geq0} c_n e^{-i (2n +1)} \psi_n  \]
We now denote by $u := \left( \pi_h f \right) \mathbf{1}_{ A_h }$ the projection of $f$ on $A_h$, then we can decompose $u$ on the finite basis $(\varphi_{n,h})_{0 \leq n \leq n_{\max}}$ of $\ell^2(A_h)$, where $n_{\max} \in \N^*$ is a fixed integer:
\[ u(a)=  \sum_{n \geq0} c_{n,h} \varphi_{n,h}(a)  \]
for all $a \in A_h$, where the scalars $(c_{n,h})_{0 \leq n \leq n_{\max}}$ are defined through the relation
\[ c_{n,h} = \langle u , \varphi_{n,h} \rangle_{\ell^2(h\Z)} = h \sum_{a \in h\Z} u(a) \varphi_{n,h}(a).  \]
Then, as in the continuous case, we can express the solution $\psi_h$ of \eqref{discrete_schrodinger} by 
\[ \psi_h(t,\cdot)= \sum_{n=0}^{n_{\max}} c_{n,h} e^{-i (2n+1) t} \varphi_{n,h}.   \]

\begin{theorem}
\label{globalbound}
Assume that $f$ is smooth. Then there exists $C$ such that for all $s$, there exists $C_s$ such that for all $h$ sufficiently small and all $n_max \in \N$ such that $\frac{1}{4}\delta |\log h| \leq n_{\max} \leq  \frac{1}{3}\delta |\log h|$, 
for all $t \geq 0$,
\[ \| \pi_h \psi(t,\cdot) - \psi_h(t, \cdot) \|_{\ell^2(h \Z)} \leq C h^{2-\delta}  + \frac{C_s}{|\log h|^s}\]
for all $\delta >0$. 
\end{theorem}
\begin{proof}
We directly compute, for all $a \in A_h$, with the notation $\lambda_n = 2n + 1$, 
\begin{align*}
\pi_h \psi(t,a) - \psi_h(t,a) & =\sum_{n \geq 0} c_n e^{-i \lambda_n t} \psi_n(a) - \sum_{n=0}^{n_{\max}} c_{n,h} e^{-i \lambda_n t} \varphi_{n,h}(a) \\
 & = \sum_{n \geq n_{\max} + 1} c_n e^{-i \lambda_n t} \psi_n(a) + \sum_{n=0}^{n_{\max}} e^{-i \lambda_n t} \left( c_n  \psi_n(a) - c_{n,h} \varphi_{n,h}(a) \right)
\end{align*}
so that 
\[  \| \pi_h \psi(t,\cdot) - \psi_h(t, \cdot) \|_{\ell^2(h \Z)} \leq  \mathcal{E}_1 + \mathcal{E}_2 + \mathcal{E}_3, \]
where 
\[ \mathcal{E}_1 := \left\| \sum_{n \geq n_{\max} +1} c_n e^{-i \lambda_n t} \psi_n(a) \right\|_{\ell^2(h \Z)}, \ \ \ \mathcal{E}_2:=  \sum_{n=0}^{n_{\max}} | c_n - c_{n,h}  | \| \psi_n \|_{\ell^2(h \Z)} \]
and
\[ \mathcal{E}_3 :=  \sum_{n=0}^{n_{\max}} |c_{n,h}  | \| \psi_n - \varphi_{n,h} \|_{\ell^2(h \Z)} .  \]
We assume that the first term is smooth in the sense that it belongs to all space $\Sigma^s(\R)$, $s\geq 0$. Classically, this is equivalent to say that the coefficients $c_n$ decay like $C_N \langle n \rangle^{-N}$ for all $N$, with constant depending on $N$. 
This shows that for all $s$, there exists $C_s$ such that 
\[
\mathcal{E}_1  \leq C_s \langle n_{\max}\rangle^{-s} \leq \frac{C_s}{|\log h|^s}
\]
as $ |\log h|\lesssim n_{\max}$. 
In order to bound $\mathcal{E}_2$, by definition of $c_n$ and $c_{n,h}$ we see that
\[ | c_{n,h} - c_n| \leq \left|\langle f, \psi_n \rangle_{L^2(\R)} - \langle u, \psi_n \rangle_{\ell^2(h\Z)} \right| + \left|  \langle u, \psi_n- \varphi_{n,h}  \rangle_{\ell^2(h\Z)}\right|,  \]
where
\[ \left|  \langle u, \psi_n- \varphi_{n,h} \rangle_{\ell^2(h\Z)} \right| \leq \| u \|_{\ell^2(h\Z)} \| \psi_n - \varphi_{n,h} \|_{\ell^2(h\Z)} \leq C_{n_{\max}} h^{2-\delta}  \| f\|_{L^2(\R)}     \]
for all $\delta >0$ by using \eqref{convphi}, and because $n_{\max} \leq \frac13 \delta |\log h|$. The error $\mathcal{E}_3$ is estimated similarly, by noticing that \[  |c_{n,h}| \leq \| u\|_{\ell^2(h\Z)} \| \varphi_{n,h} \|_{\ell^2(h\Z)} \leq C_{n_{\max}} \| f\|_{H^1(\R)},   \]
so $\mathcal{E}_3 = \mathcal{O}\left(h^{2- \delta} \right)$.

\end{proof}
\begin{remark}
The error term in $\log h$ could be refined by a better frequency estimate in \eqref{convphi}, like for example $n \leq h^{-\delta}$ but this would require a better bound for the asymptotics of the Kravchuk functions which is out of the scope of this work. 
\end{remark}

\section{Kravchuk Transform} \label{kravchuk_transform_section}
We now prove Theorem \ref{Theorem3}, which can also be found in \cite{atakishiyev1997}. 
For a vector $x \in \R^{N+1}$, we define the transformation $\tilde{x}=\mathcal{K} x$ by the formula
\[ \tilde{x}_k= \sum_{j=0}^{N} e^{i\frac{\pi}{2} (j-k-N/2)} \phi_k(j) x_j  \]
for all $0 \leq k \leq N$, corresponding to the multiplication by the matrix \begin{multline*}
 K=e^{-\frac{i\pi N}{4}} \times \\
 \left( \begin{array}{ccccc}  
					 		\phi_0(0) & e^{\frac{i\pi}{2}} \phi_0(1) & \hdots & e^{\frac{i\pi(N-1)}{2}} \phi_0(N-1)& e^{\frac{i\pi N}{2}} \phi_0(N) \\
							e^{-\frac{i\pi}{2}} \phi_1(0) & \phi_1(1) & \hdots & e^{\frac{i\pi(N-2)}{2}} \phi_1(N-1) & e^{\frac{i\pi (N-1)}{2}} \phi_1(N) \\
							\vdots & \vdots & \ddots & \vdots & \vdots \\
							e^{-\frac{i\pi N}{2}} \phi_N(0)  & e^{-\frac{i\pi(N-1)}{2}} \phi_N(1) & \hdots & e^{-\frac{i\pi}{2}} \phi_N(N-1) & \phi_N(N) 
				 		\end{array} \right).  
\end{multline*}
Recall that here $\phi_n(k)$ denote the functions \eqref{kravfon} corresponding to the function $\varphi_{n,h}$ after scaling. 
With the notation \eqref{eqL} and \eqref{eqDA}, we obtain
\[ K= e^{-\frac{i\pi N}{4}} D^* L D. \]
As a direct consequence of the fact that the $(\phi_n)_{0 \leq n \leq N}$ are orthogonal, we get that the matrix $L$ and $K$ are unitary: 
\[K^* K=L^* L = \Id. \]
%
%
%\begin{proposition}
%The matrix $F$ satisfies 
%\[ K^2= \left(   \begin{array}{ccc}  
%					  (0) &  & 1 \\
%					  & \iddots &  \\
%					 1&  & (0) 
%				 \end{array} \right) \ \ \ \text{and} \ \ \  K^4= \Id .  \]	
%\end{proposition}
%\begin{proof}
%We first remark that
%\[ K^2= e^{-\frac{i\pi N}{2}} D^* L^2 D,  \]
%with, for $1 \leq i,j \leq N+1$,
%\[ (L^2)_{i-1,j-1}= \sum_{k=0}^N \phi_i(k) \phi_k(j)= \sum_{k=0}^N (-1)^k \phi_i(k) \phi_k(N-j)= (-1)^{N-j} \sum_{k=0}^N \phi_i(k) \phi_{N-j}(k)  \]	
%using both the parity property Proposition \ref{prop_parity_kravchuk} and the symmetry property \ref{prop_symmetry_kravchuk_functions}, hence 
%\[ (L^2)_{i-1,j-1}= \delta_{i+j,N} (-1)^{i} \]
%by orthogonality property Proposition \ref{prop_orthogonality_kravchuk_functions}, and we get the result.
%\end{proof}
%
%
%We are now going to see that $K$ can be written in a special form. We denote $\alpha=(\alpha_1,\ldots,\alpha_N)$ such that
%\[ \alpha_k=\sqrt{k \left( N-k+1\right)}  \]
%for all $1 \leq k \leq N$, and
%	\[ A= \left(   \begin{array}{cccc}  
%					  N+1 & -\alpha_1 & & (0) \\
%					 -\alpha_1  & N+1 & \ddots &  \\
%					   & \ddots & \ddots & -\alpha_N  \\
%					 (0) &  & -\alpha_N & N+1
%				 \end{array} \right).    \]	
%In particular, $A \in \mathcal{S}_{N+1}(\R)$ and is tridiagonal. Then we have the following key property:
\begin{proposition}
The matrix $K$ satisfies
\[  K=e^{\frac{i \pi}{4}} e^{-\frac{i \pi}{4} A},    \]
where $A$ is the matrix defined in \eqref{eqDA}. 
\end{proposition}
\begin{proof}
We are going to show that both matrices $K^*$ and $e^{-\frac{i \pi}{4}} e^{\frac{i \pi}{4}A}$ have the same image on a particular basis of $\R^{N+1}$, the basis formed by the $N+1$-vectors $v_n:=(\phi_n(0),\phi_n(1),\ldots,\phi_n(N))^{\top}$ for $0 \leq n \leq N$. We first compute that for all $0\leq k \leq N$,
\[ (A v_n)_k= (N+1) \phi_n(k) - \sqrt{k(N-k+1)} \phi_n(k-1) - \sqrt{(k+1)(N-k)} \phi_n(k+1), \]
hence from Proposition \ref{prop_eigenfunctions_X_N} we get that
\[ (A v_n)_k= \left( 2n +1 \right) \phi_n(k), \]
so the matrix A is diagonal in the basis $(v_0, v_1, \ldots, v_n)$, and
\begin{align*}
 e^{-\frac{i \pi}{4}} e^{\frac{i \pi}{4}A} v_n  
 & = e^{-\frac{i \pi}{4}} \sum_{k \geq 0} \frac{1}{k!} \left( \frac{i \pi}{4} \right)^k \left( 2n+1 \right)^k v_n  = e^{-\frac{i \pi}{4}} e^{ \frac{i \pi}{2} \left(n+\frac{1}{2} \right)} v_n \\
 & =  e^{i\frac{n \pi}{2}} v_n.
\end{align*}
On the other hand, using Proposition \ref{prop_transform_kravchuk_functions}, we compute
\[
  \left( K^* v_n  \right)_m  = \left( e^{\frac{i \pi N}{4}} D^* L^* D v_n \right)_m = e^{\frac{i \pi N}{4}} \sum_{k=0}^N \phi_n(k) \phi_k(m) e^{i \frac{k-m}{2} \pi}  = e^{i \frac{n}{2} \pi} \phi_n(m),
\]
so 
\[K^*=e^{-\frac{i \pi}{4}} e^{\frac{i \pi}{4}A}\]
and we get the result.
\end{proof}

%
%\subsection{Convergence towards Hermite coefficients}
%Let's decompose a function $f \in H^1(\R)$ on the Hermite function basis $(\psi_n)_{n \in \N}$, namely
%\[ f(x)=\sum_{n \geq 0} c_n \psi_n(x)   \]
%for all $x \in \R$. We now project $f$ on the grid $h \left[-N/2, N/2 \right]$ with 
%\[ h=\sqrt{\frac{2}{N}}, \]
%and we denote
%\[ c_{n,h}=\sum_{j=0}^N \phi_k \left(j- \frac{N}{2}  \right) f\left(h \left(j- N/2 \right)\right),    \]
%the coefficient of $f_{|h \Z}$ in the Kravchuk basis. Then using Theorem \ref{theorem_kravchuk_functions},  we have
%\[   \sqrt{h} c_{n,h} \rightarrow c_n. \]
%Note that we also have $L$ unitary, $L^2=K^2$ and $L^4=\Id$.
%
%\subsection{Fast Kravchuk Transform} 
%The matrix $K$ is the exponential of a skew-symmetric tridiagonal matrix: super-fast algorithm ?

\section{Numerical simulations of Kravchuk functions} \label{numerics}

In this section we are going to present some plots and numerical simulations of some of the previous properties of the Kravchuk functions $(\varphi_{n,h})_n$. First we illustrate the convergence result \eqref{convrho} on the grid $X_N$ for $N=50$. We plot both functions $\rho_h$ and $\rho/\sqrt{\pi}$ in Figure \ref{fig:binomial_to_gaussian.png}. We observe than even if the number of space discretization points is rather low, the accuracy of the approximation of the Gaussian function is pretty good. This is highlighted by Figure \ref{fig:error_binomial_to_gaussian.png}, where we plot the $\ell^2(h \Z)$, the $\ell^{\infty}(h \Z)$ and the $h^1(h \Z)$ norms\footnote{The $h^1(h \Z)$ norm is defined by $\Norm{u}{h^1(h\Z)}^2 = \Norm{v}{\ell^2(h\Z}^2 + \langle v , \Delta_h v\rangle_{\ell^2(h\Z)}$ where $\Delta_h$ is the discrete Laplacian. } of the difference $\rho_h - \rho/\sqrt{\pi}$ in logarithmic scale with respect to $N=2/h^2$, and we get a leading coefficient of $-1$ for these three lines, which corresponds to the $O(h^{2-\delta})$ for $\delta>0$ as small as we want in \eqref{convrho}. 

\begin{figure}[h]
	\centering
		\includegraphics[width=0.95\textwidth,trim = 0cm 0cm 0cm 0cm, clip]{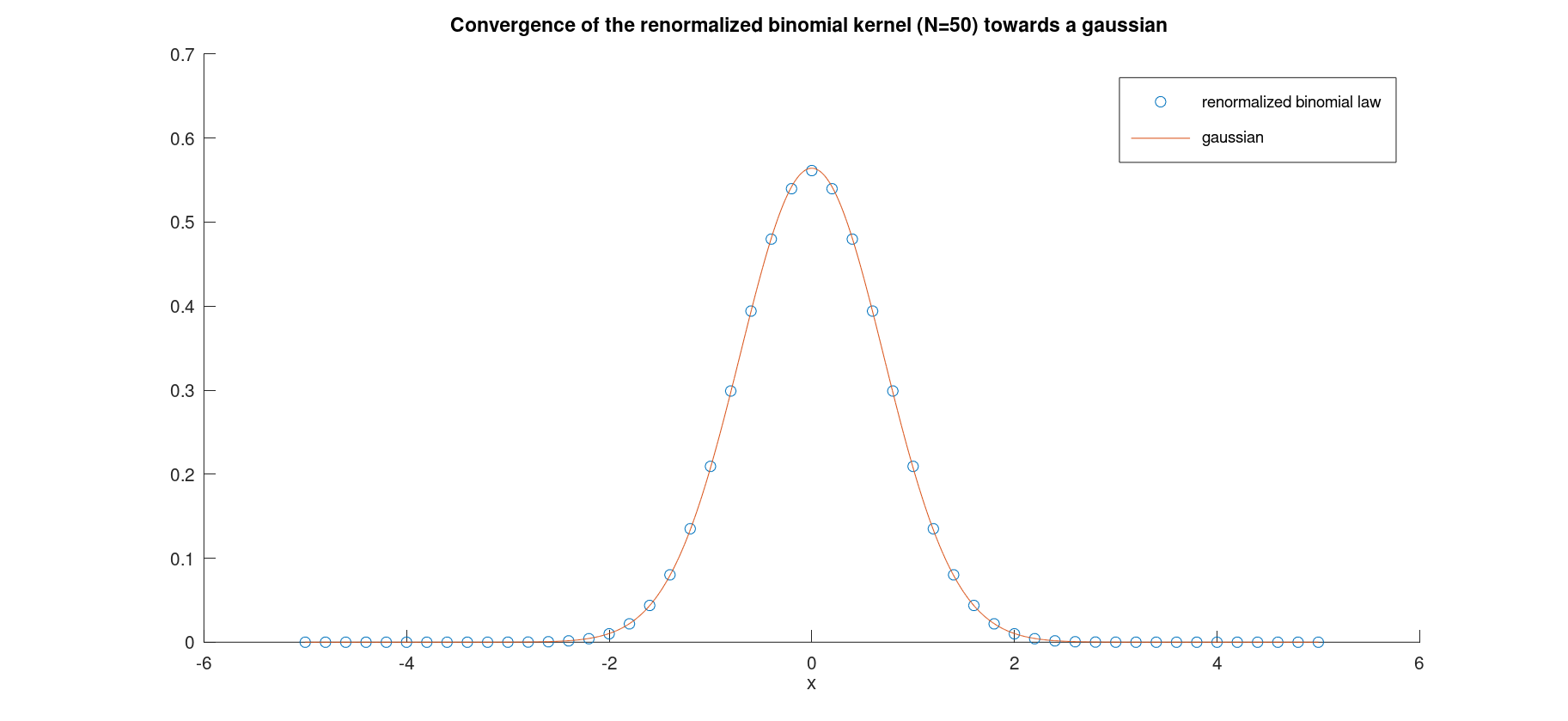}	
	\caption{Convergence of the binomial law $\rho_h$ to the Gaussian $\rho/\sqrt{\pi}$.}
	\label{fig:binomial_to_gaussian.png}
\end{figure}

\begin{figure}[h]
	\centering
		\includegraphics[width=0.95\textwidth,trim = 0cm 0cm 0cm 0cm, clip]{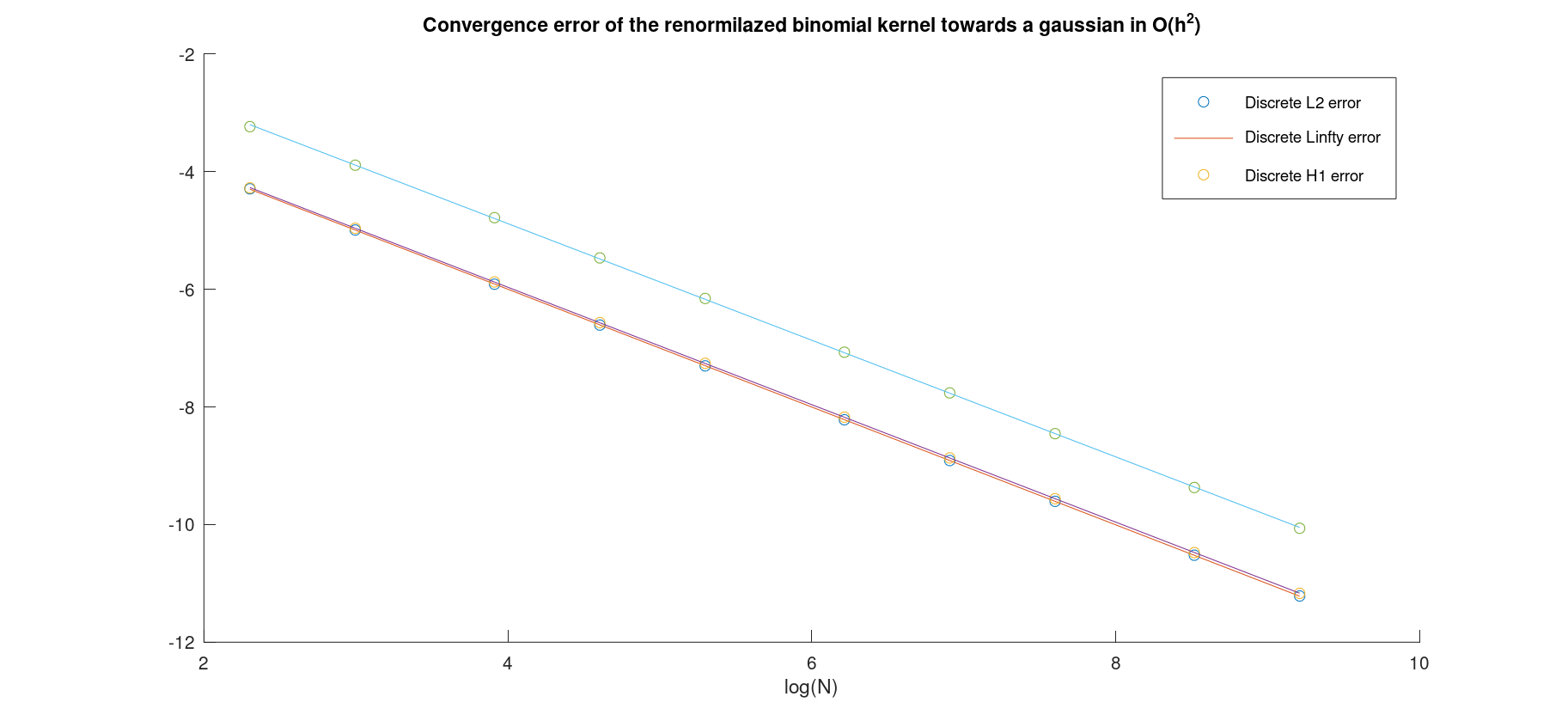}	
	\caption{$\ell^2(h \Z)$, $\ell^{\infty}(h \Z)$ and $h^1(h \Z)$ convergence error of the binomial law $\rho_h$ to the Gaussian $\rho/\sqrt{\pi}$.}
	\label{fig:error_binomial_to_gaussian.png}
\end{figure}

In the same way, we illustrate Theorem \ref{theorem_kravchuk_oscillator} by plotting $(\varphi_{n,h})_{1 \leq n \leq 6}$ with $N=50$ and comparing these functions to the first Hermite functions $(\psi_n)_{1 \leq n \leq 6}$ in Figure \ref{fig:kravchuk_to_hermite.png}, then by computing the $\ell^2(h \Z)$, $\ell^{\infty}(h \Z)$ and $h^1(h \Z)$ errors of the differences $\varphi_{10,h} - \psi_{10}$ for $n=10$ in logarithmic scale, where we observe a convergence in $-\log(N)$ which well corresponds to the $O(h^{2-\delta})$ for $\delta>0$ of Theorem \ref{theorem_kravchuk_oscillator}.

\begin{figure}[h]
	\centering
		\includegraphics[width=0.95\textwidth,trim = 0cm 0cm 0cm 0cm, clip]{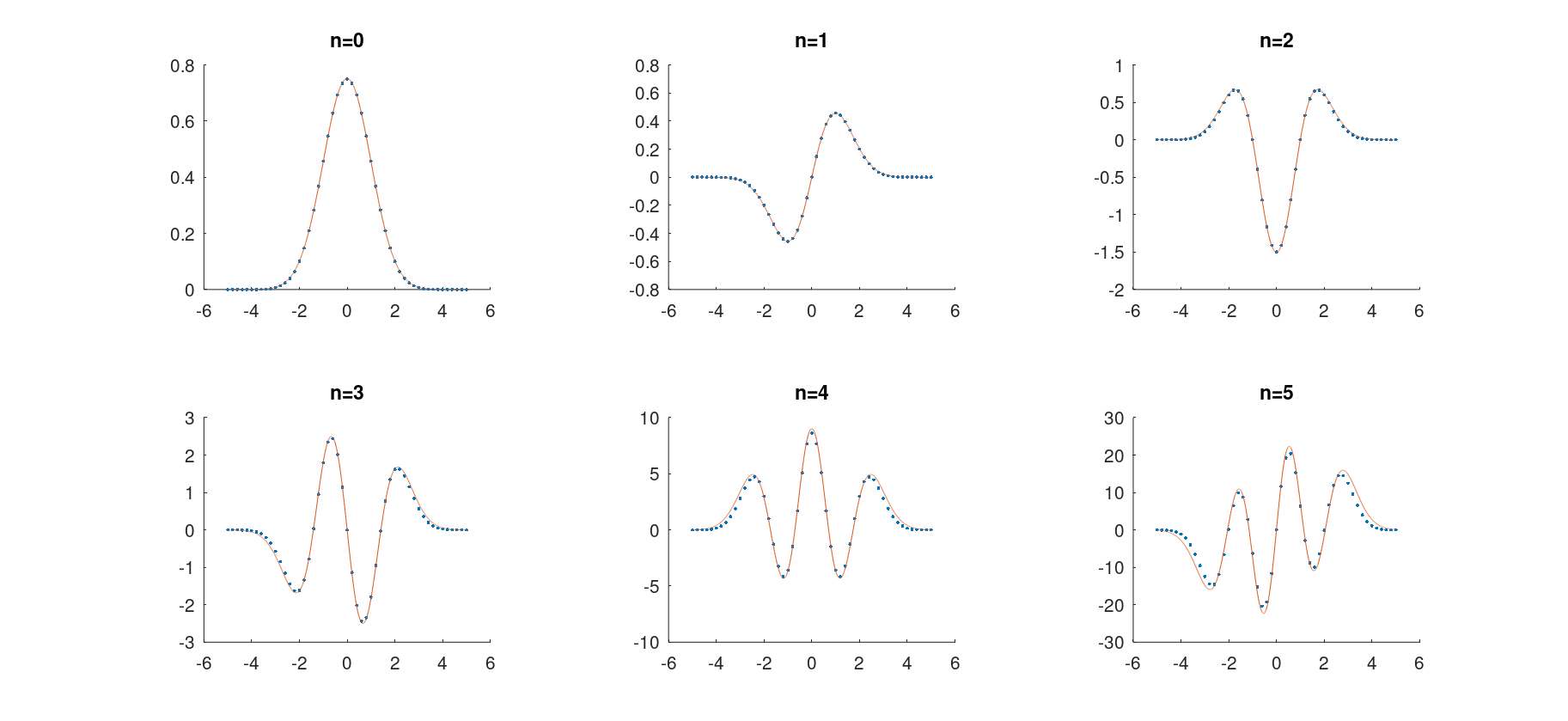}	
	\caption{Convergence of the Kravchuk functions $(\varphi_{n,h})_{1 \leq n \leq 6}$ to the Hermite functions $(\psi_n)_{1 \leq n \leq 6}$.}
	\label{fig:kravchuk_to_hermite.png}
\end{figure}

\begin{figure}[h]
	\centering
		\includegraphics[width=0.95\textwidth,trim = 0cm 0cm 0cm 0cm, clip]{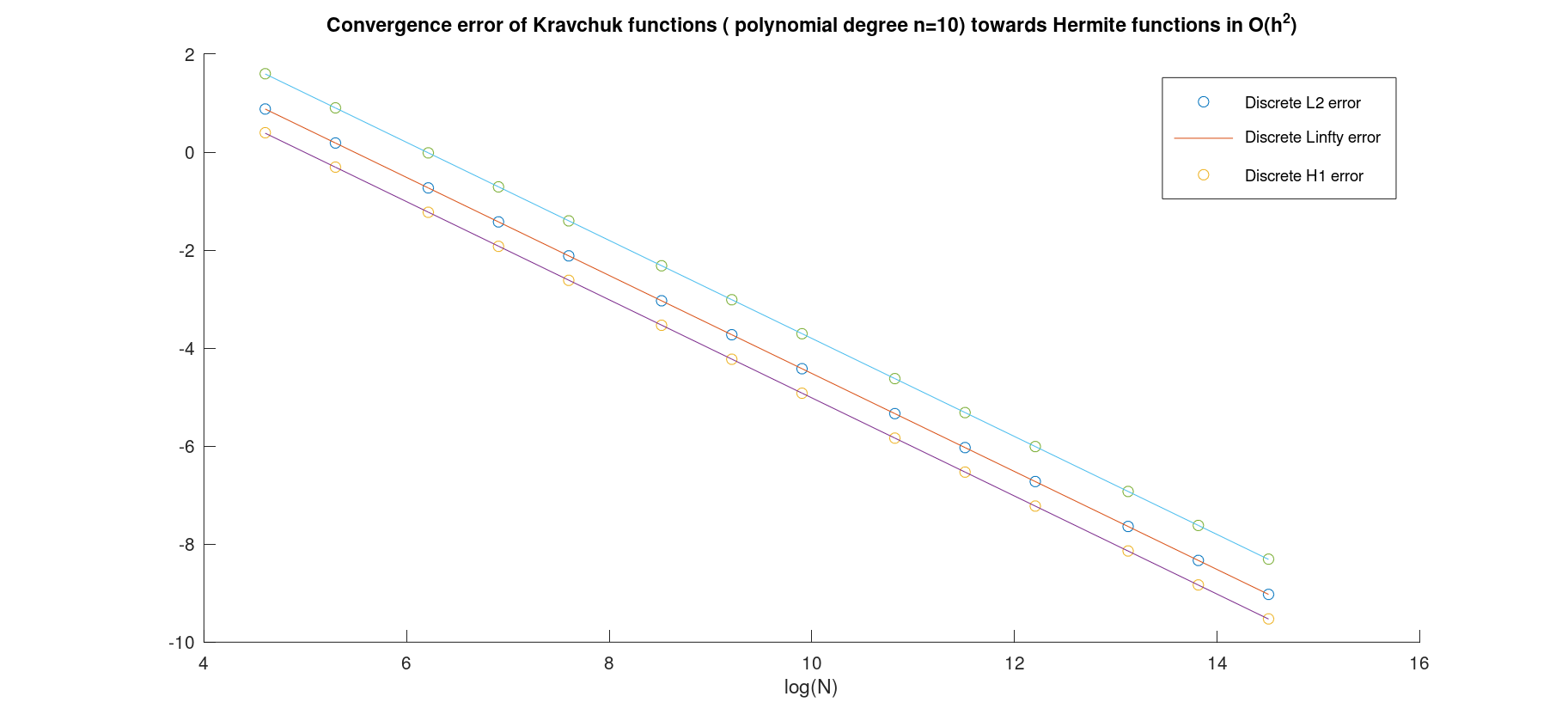}	
	\caption{$\ell^2(h \Z)$, $\ell^{\infty}(h \Z)$ and $h^1(h \Z)$ convergence error of the Kravchuk functions $(\varphi_{n,h})_{1 \leq n \leq 6}$ to the Hermite functions $(\psi_n)_{1 \leq n \leq 6}$.}
	\label{fig:error_kravchuk_to_hermite.png}
\end{figure}

\subsection*{Acknowledgments} Q.C. is supported by the Labex CEMPI (ANR-11-LABX-0007-01).

\bibliographystyle{siam}
\bibliography{biblio}

\end{document}